\documentclass[11pt]{amsart}
\usepackage[paperwidth=210mm,paperheight=280mm,textwidth=160.9mm,textheight=225mm,centering]{geometry}
\usepackage{amssymb,amsmath,latexsym,lipsum}
\usepackage{enumitem}
\usepackage{tabulary}
\usepackage{booktabs}
\usepackage{ upgreek }
\usepackage{charter}
\usepackage[title]{appendix}
\usepackage{longtable}
\usepackage{amsthm}
\usepackage{euscript}
\usepackage{color}
\usepackage[T1]{fontenc}
\usepackage[utf8]{inputenc}
\usepackage[singlespacing]{setspace}
\usepackage[colorinlistoftodos]{todonotes}
\usepackage{tikz}
\usetikzlibrary {positioning}
\definecolor {processblue}{cmyk}{0.96,0,0,0}
\usepackage{bookmark}
\usepackage{multirow,bigdelim}
\usepackage{stackengine}
\usepackage{mathrsfs}
\usepackage{mathtools}
\usepackage{tikz-cd}
\usepackage{pgfplots}
\usepackage{pgf}
\usetikzlibrary{arrows}

\usepackage{caption} 
\usepackage{subcaption}

\newtheorem{theorem}{Theorem}[section]
\newtheorem{proposition}[theorem]{Proposition}
\newtheorem{lemma}[theorem]{Lemma}
\newtheorem{corollary}[theorem]{Corollary}
\newtheorem{definition}[theorem]{Definition}
\theoremstyle{remark}
\newtheorem{example}[theorem]{Example}
\theoremstyle{remark}
\newtheorem{remark}[theorem]{Remark}
\newtheorem{question}[theorem]{Question}

\usepackage{etoolbox}

\makeatletter
\def\l@section{\@tocline{1}{4pt plus1pt}{0pt}{2.1em}{\bfseries}}
\def\l@subsection{\@tocline{2}{1pt}{1.3em}{3em}{\normalfont}}
\patchcmd{\@tocline}{\hfil\hbox}{%
  \leaders\hbox{\normalfont\kern.25em.\kern.25em}\hfill\hbox
}{}{\PackageWarning{manuscript}{Contents leaders could not be applied}}
\makeatother

\DeclareMathOperator{\spec}{Spec}

\DeclareMathOperator{\spf}{Spf}

\DeclareMathOperator{\ord}{ord}

\DeclareMathOperator{\an}{an}

\DeclareMathOperator{\sw}{sw}

\DeclareMathOperator{\Frac}{Frac}

\DeclareMathOperator{\dsw}{dsw}

\DeclareMathOperator{\st}{st}

\DeclareMathOperator{\Proj}{Proj}

\DeclareMathOperator{\cohom}{H}

\pgfplotsset{compat=1.16}

\begin{document}

\title{The Hurwitz tree obstruction for the refined local lifting problem}

\author{Huy Dang}
\email{hdang2@binghamton.edu}
\address{Department of Mathematics and Statistics, Binghamton University (SUNY), Binghamton, New York, 13901}


\begin{abstract}
We prove that the differential Hurwitz tree obstruction for the cyclic $p$-power refined local lifting problem always vanishes. Given a cyclic extension of degree $p^n$ in characteristic $p$ and a lift of its degree-$p^{n-1}$ subextension, we construct a compatible differential Hurwitz tree extending the tree of that lift. Realizing the resulting tree by a cover remains a separate problem.
\end{abstract}

\maketitle

\begingroup
\small
\tableofcontents
\endgroup

\section{Introduction}
\label{secintro}

The lifting problem asks whether a curve defined over a field $k$ with a faithful group action can be lifted to characteristic zero \cite{MR3051249}. Grothendieck showed that curves with a trivial action can be lifted to the ring of Witt vectors $W(k)$ \cite[III]{MR0217087}. To understand the liftability of curves with a nontrivial action, one can reduce the problem to understanding Galois extensions of the ring of power series via a local-to-global principle \cite{MR1767273}.

It was conjectured by Oort \cite{MR927980} that all cyclic extensions lift to characteristic $0$. The first result was established by Oort, Sekiguchi, and Suwa, who demonstrated that the conjecture holds for all $\mathbb{Z}/pm$-covers, where $(m,p)=1$ \cite{MR1011987}. Subsequently, Green and Matignon proved the case when $G \cong \mathbb{Z}/p^2m$ \cite{MR1645000}. Obus--Wewers and Pop resolved the conjecture for cyclic groups in 2014 \cite{MR3194815, MR3194816}. 

Sa{\"i}di's \emph{refined lifting conjecture} \cite[Conj-O-Rev]{MR3051252} predicts the liftability of cyclic coverings of curves in towers. An affirmative answer to the following local question implies this conjecture.

\begin{question}
\label{questionrefinedlifting}
Let $k[[z]]/k[[x]]$ be a $G$-Galois extension where $G$ is cyclic and $H$ is a quotient of $G$. Suppose we are given a discrete valuation ring $R$ in characteristic zero and a lift $R[[S]]/R[[X]]$ of the $H$-subextension $k[[s]]/k[[x]]$. Does there exist a finite extension $R'$ of $R$ in characteristic zero with residue field $k$ and a $G$-Galois extension $R'[[Z]]/R'[[X]]$ that lifts $k[[z]]/k[[x]]$ and contains $R'[[S]]/R'[[X]]$ as the $H$-subextension? Equivalently, can one fill in the following commutative diagram, where $\phi_G,\phi_H$ denote the given covers and $\Phi_G,\Phi_H$ their lifts?
\begin{equation*}
\begin{tikzcd}
\spec k \arrow[d] & \spec k[[x]] \arrow[d] \arrow[l] & \spec k[[s]] \arrow[l, "\phi_H"] \arrow[d] & \spec k[[z]] \arrow[l] \arrow[d, dotted] \arrow[ll, bend right=20, "\phi_G"] \\
\spec R'   & \spec R'[[X]] \arrow[l] & \spec R'[[S]] \arrow[l, "\Phi_H"] & \spec R'[[Z]] \arrow[l, dotted] \arrow[ll, dotted, bend left=20, "\Phi_G",labels=above]
\end{tikzcd}
\end{equation*}
The vertical closed immersions identifying the special fibers are equivariant for the respective $G$- and $H$-actions. The maps from the $G$-covers to the $H$-covers are equivariant with respect to the quotient homomorphism $G\to H$.
\end{question}

\begin{remark}
    The rigid generic fiber of the formal scheme $\spf R'[[X]]$ is an open unit disc $D$ in the rigid analytic projective line $\mathbb{P}^{1, {\rm an}}_{\Frac(R')}$. This allows us to study the degeneration of covers of $D$ using non-archimedean analysis.
\end{remark}

A \emph{Hurwitz tree} records the degeneration of a cover of a rigid disc: its underlying tree comes from a semistable model, and its vertices carry depth and differential Swan conductors. Henrio introduced this technique for the order-$p$ lifting problem \cite{2000math.....11098H}; subsequent developments include \cite{MR2254623,MR2534115}.

\begin{remark}
    In equal characteristic, Hurwitz trees yield the refined lifting result \cite[Theorem 1.2]{2020arXiv201013614D} and describe the geometry of the $p$-fiber of moduli spaces of cyclic torsors \cite{10.1093/imrn/rnae060}.
\end{remark}

In this paper, we define the \emph{differential Hurwitz tree obstruction} for the refined local lifting problem and prove that it always vanishes for cyclic $p$-power covers.

\begin{theorem}
\label{theoremmain}
Let $n\geq2$. In the setting of Question \ref{questionrefinedHKG}, let $\mathcal T_{n-1}$ be the differential Hurwitz tree of the prescribed lift $\Phi_{n-1}$. There exists a compatible {\'e}tale $\mathbb Z/p^n$ differential Hurwitz tree $\mathcal T_n$ extending $\mathcal T_{n-1}$ whose root reduction type represents the prescribed local extension $k[[z_n]]/k[[x]]$. In particular, the differential Hurwitz tree obstruction for the cyclic $p$-power refined local lifting problem always vanishes.
\end{theorem}

Realizing this abstract differential tree by a cover remains the next step in Question \ref{questionrefinedlifting}.

\subsection{Outline}
Section \ref{secdegen} recalls refined Swan conductors and good reduction. Section \ref{secHurwitz} defines compatible differential Hurwitz trees and their extensions. Section \ref{secvanishofHurwitztree} gives examples, and Section \ref{sectechnical} proves Theorem \ref{theoremmain} by two Cartier constructions and a recursive extension argument.

\subsection{Notations and conventions}
$K$ will always denote a field of characteristic $0$ that is complete with respect to a discrete valuation $\nu: K^{\times} \rightarrow \mathbb{Q}$, normalized by $\nu(p)=1$. The residue field $k$ of $K$ is algebraically closed of characteristic $p>0$. One example is $\Frac W(k)$, the fraction field of the ring of Witt vectors over $k$. The ring of integers of $K$ will be denoted by $R$.

We fix an algebraic closure $\overline{K}$ of $K$, and whenever necessary, we will replace $K$ by a suitable finite extension within $\overline{K}$, without changing the above notation. The symbol $\mathbb{K}$ usually denotes a function field over $K$ (e.g., $\mathbb{K}=K(X)$).

\subsection{Acknowledgements}
The author expresses gratitude to the National Center for Theoretical Sciences (Taiwan) and Binghamton University for providing excellent working conditions. He thanks Andrew Obus, Stefan Wewers, Florian Pop, and David Harbater for useful discussions.

The proofs of Proposition \ref{theoremminimalcartiergeneral} and Lemma \ref{lemcongruentsubminimalcartier} were suggested by ChatGPT versions 5.6 and 6. The author also used AI to check the grammar of the manuscript. The author has checked all AI-assisted material and takes full responsibility for the content of this paper.

\section{Degeneration of cyclic characters}
\label{secdegen}

\subsection{Cyclic covers of the projective line}
\label{secASWprojective}
Recall that every $\mathbb{Z}/p^n$-cover $\phi_n$ of the projective line $\mathbb{P}^1_k=\Proj k[x,w]$ can be represented by the class of a length-$n$ Witt vector over $k(x)$ $$\underline{g}=(g^1, \ldots, g^n) \in W_n(k(x))/\wp(W_n(k(x)))\cong \cohom^1(k(x), \mathbb{Z}/p^n).$$
Here $W_n(k(x))$ is the ring of length-$n$ Witt vectors, and $\wp=F-\mathrm{id}$ is the Artin--Schreier--Witt operator, with $F$ the Witt Frobenius. For characters, the power $\chi^{p^a}$ denotes the multiple $p^a\chi$ in the additive cohomology group. If $\chi$ has order $p^n$ and $0\leq a<n$, we regard the resulting quotient character as taking values in $\mathbb Z/p^{n-a}$.
We identify $\phi_n$ with a character in $\cohom^1(k(x), \mathbb{Z}/p^n)$. We say $\underline{g}$ is \emph{reduced} if, in the partial fraction decomposition of each entry, every nonzero pole order is prime to $p$, every nonconstant monomial in the polynomial part has degree prime to $p$, and the constant term is zero. Since $k$ is algebraically closed, every class in $W_n(k(x))/\wp(W_n(k(x)))$ has a unique representative in this reduced form. Moreover, for $1 \leq i < n$, the character $\phi_i:=(\phi_n)^{p^{n-i}}$ corresponds to the unique $\mathbb{Z}/p^i$-subcover $\phi_i: Y_i \to \mathbb{P}_k^1$ of $\phi_n$ represented by $(g^1, \ldots, g^i)$.

\begin{remark}
\label{remarkramification}    
Write $\mathbb B(\phi_n)$ for the branch locus of $\phi_n$; we use the same notation for the branch locus of the cover associated with a character. Let $Q_j$ lie above $P_j\in\mathbb B(\phi_n)$, and suppose its inertia group is $\mathbb Z/p^m$, with $m\leq n$. Since $k$ is algebraically closed, the induced extension of completed local rings $\hat{\mathcal O}_{Y_n,Q_j}/\hat{\mathcal O}_{\mathbb P^1,P_j}$ is $\mathbb Z/p^m$-Galois. We say the \emph{$i$-th ramification break} of $\phi_n$ at $P_j$, denoted by $m_{j,i}$, is $-1$ for $i \le n-m$. When $i > n-m$, the $i$-th ramification break of $\phi_n$ at $P_j$ is the $(i-n+m)$-th one of $\hat{\mathcal{O}}_{Y_n,Q_j}/\hat{\mathcal{O}}_{\mathbb{P}^1,P_j}$. We call $(m_{j,n-m+1}, \ldots, m_{j,n})$ the sequence of \emph{upper ramification breaks} of $\phi_n$ at $P_j$. The following result computes these invariants from the reduced representation.
\end{remark}

\begin{theorem}[{\cite[Theorem 1]{MR1935414}}]
\label{theoremcaljumpirred}
With the notation above, suppose that $\underline{g}$ is reduced. Then we have
\begin{equation}
\label{eqnformulalowerjumpasw}
    m_{j,i}=\max\{ p^{i-l} \deg_{(x-P_j)^{-1}} (g^{l}) \mid l=1, \ldots, i\}, 
\end{equation}
for $i>n-m$. Here $\deg_{(x-P_j)^{-1}}(g^l)$ denotes the pole order of $g^l$ at $P_j$, taken to be $0$ if $g^l$ is regular there. At $P_j=\infty$, use the local parameter $x^{-1}$ instead of $x-P_j$.
\end{theorem}

\subsection{Set up}
\label{secsetup}
Suppose we are given a $\mathbb{Z}/p^{n}$-extension $\phi_n:k[[x]] \xrightarrow{} k[[z_n]]$ as in Question \ref{questionrefinedlifting}, where $k$ is an algebraically closed field of characteristic $p$. One may associate with $\phi_n$ a unique \emph{Katz--Gabber cover}, also called a \emph{Harbater--Katz--Gabber (HKG) cover} (see \cite{MR579791} and \cite{MR867916}), $W_n \xrightarrow[]{\phi_n} \mathbb{P}^1_k \cong \Proj k[x,w]$ that is
\begin{itemize}
    \item totally ramified above $x=0$ and \'etale away from $0$,
    \item the formal completion above $0$ yields $k[[z_n]]/k[[x]]$.
\end{itemize}
By the discussion in \S \ref{secASWprojective}, we may assume that the extension $\phi_n$ is given by a \emph{reduced} Witt vector of length $n$ over $k(x)$
\begin{equation*}
    \underline{g}=(g^1, \ldots, g^n) \in W_n(k(x)).
\end{equation*}
Additionally, due to Remark \ref{remarkramification}, we may further assume that $g^i \in k[x^{-1}]$ for $1 \le i \le n$.

By the local--global lifting principle \cite[Proposition 1.2.4]{MR3051252}, Question \ref{questionrefinedlifting} is equivalent, after a finite extension $R'/R$, to the following formulation for one-point covers.

\begin{question}
    \label{questionrefinedHKG}
    Suppose $k[[z_n]]/k[[x]]$ is a $\mathbb{Z}/p^n$-Galois extension, and $\Phi_{n-1}: Y_{n-1} \xrightarrow{} C := \mathbb{P}^1_K$ is a $\mathbb{Z}/p^{n-1}$-cover with the following properties.
\begin{enumerate}[label=(\arabic*)]
    \item \label{propmainonepointcover1} The cover $\Phi_{n-1}$ has good reduction with respect to the standard model $\mathbb{P}^1_R$ of $C$ and reduces to a $\mathbb{Z}/p^{n-1}$-cover $\phi_{n-1}: \overline{Y}_{n-1} \xrightarrow{} \overline{C} \cong \mathbb{P}^1_k$ that is totally ramified above $x=0$ and {\'e}tale elsewhere.
    \item \label{propmainonepointcover2} The completion of $\phi_{n-1}$ at $x=0$ yields $k[[z_{n-1}]]/k[[x]]$, which is the unique $\mathbb{Z}/p^{n-1}$-sub-extension of the given extension $k[[z_n]]/k[[x]]$.
\end{enumerate}
Does $\Phi_{n-1}$ extend to a $\mathbb{Z}/p^n$-cover $\Phi_n: Y_n \xrightarrow{} C$ with good reduction such that
\begin{enumerate}[label=(\alph*)]
    \item \label{propmainonepointcover3} Its reduction $\phi_n: \overline{Y}_n \xrightarrow{} \overline{C}$ is totally ramified above $\overline{x}=0$, {\'e}tale everywhere else.
    \item \label{propmainonepointcover4} The completion of $\phi_{n}$ at $x=0$ yields $k[[z_{n}]]/k[[x]]$.
\end{enumerate}
\end{question}

We call a cover such as $\Phi_{n-1}$, where all branch points on the generic fiber reduce to a single point on the special fiber, an \emph{admissible cover}. The associated character is called an \emph{admissible character}.


\subsection{Discs and annuli}
\label{secdiscannuli}
Let $\epsilon \in \mathbb{Q}_{\ge 0}$, $r=p^{-\epsilon}$, $z \in R$, and choose $a \in K$ with $\nu(a)=\epsilon$, after a finite extension of $K$ if necessary. Table \ref{tabnonarchimedean} records the rings associated with the rigid analytic domains used below. The open disc is $D[\epsilon,z]:=\{u\in(\mathbb A^1_K)^{\an}\mid\nu(u-z)>\epsilon\}$. For discs centered at $0$, we abbreviate $D[\epsilon]:=D[\epsilon,0]$ and $\mathcal D[\epsilon]:=\mathcal D[\epsilon,0]$. For more details, see \cite{MR1202394}, \cite{MR774362}, \cite{MR767194}.

\begin{table}[ht]
\small
    \centering
\begin{tabular}{ |p{4.8cm}|p{3.8cm}|p{5.8cm}| }
\hline
Notion & Associated ring & Geometry  \\
 \hline
 \hline
Open unit disc & $R[[X]]$  & $D=\{ u \in (\mathbb{A}^1_K)^{\an} \mid \nu(u)>0 \}$ \\
\hline
Closed unit disc & $R\{X\}$ & $\mathcal{D}=\{ u \in (\mathbb{A}^1_K)^{\an} \mid \nu(u) \ge 0 \}$ \\
\hline
Boundary of a unit disc & $R[[X]]\{X^{-1}\}$ & $\{ u \in (\mathbb{A}^1_K)^{\an} \mid \nu(u) = 0 \}$ \\
\hline
Closed disc of radius $r$ center $z$ & $R\{a^{-1}(X-z)\}$  & $\mathcal{D}[\epsilon,z]=\{ u \in (\mathbb{A}^1_K)^{\an} \mid \nu(u-z) \ge \epsilon \}$ \\
\hline
Open disc of radius $r$ center $0$ & $R[[a^{-1}X]]$ & $D[\epsilon]:=D[\epsilon,0]$ \\
\hline
Open annulus of thickness $\epsilon$ & $R[[X,U]]/(XU-a)$ & $A(0, \epsilon)=\{ u \in (\mathbb{A}^1_K)^{\an} \mid 0< \nu(u) < \epsilon \}$ \\
\hline
\end{tabular}
\vspace{3mm}
    \caption{Non-archimedean geometry notions}
    \label{tabnonarchimedean}
\end{table}
Recall that $R\{X\} = \big\{ \sum_{i \ge 0} a_i X^i \in R[[X]] \mid \lim_{i \to \infty} \nu(a_i)=\infty  \big\}$. For $z\in R$ and $s\in\mathbb Q_{\geq0}$, the Gauss valuation associated with $\mathcal D[s,z]$ is defined on nonzero polynomials by
\begin{equation*}
    \nu_{s,z}\!\left(\sum_i a_i(X-z)^i\right)=\min_i\{\nu(a_i)+is\},
\end{equation*}
and extends to $\mathbb K=K(X)$ by $\nu_{s,z}(g/h)=\nu_{s,z}(g)-\nu_{s,z}(h)$. We denote its residue field by $\kappa_{s,z}$; it is the function field of the canonical reduction $\overline{\mathcal D}[s,z]$.

\subsection{Semi-stable models and a partition of a disc}
\label{secsemistablemodel}
Consider the open unit disc $D \subset C^{\an} \cong (\mathbb{P}^1_K)^{\an}$, the rigid generic fiber of $\spf R[[X]]$ for a coordinate $X \in \mathbb{K}$. Suppose we are given $x_{1,K}, \ldots, x_{r,K}$ in $D(K)$, with $r \ge 2$. We can think of $x_{1,K}, \ldots, x_{r,K}$ as elements of the maximal ideal of $R$. Let $C^{\st}$ be a blow-up of the standard model $C_R:=\mathbb P^1_R$ such that
\begin{itemize}
    \item the exceptional divisor $\overline{C}$ of the blow-up is a semi-stable curve over $k$,
    \item the chosen points $x_{b,K}$ specialize to pairwise distinct smooth points $x_b$ on $\overline{C}$, and
    \item $(\overline C,(x_b),x_0)$ is stably marked, where $x_0\in\overline C$ is the unique point at which $\overline C$ meets the strict transform of the original special fiber $\mathbb P^1_k$. This point is denoted here by $\overline\infty$.
\end{itemize}
The point $x_0$ is a node of the full special fiber, but a smooth point of the exceptional curve $\overline C$; it records the direction toward the outer boundary of $D$. A curve is stably marked if each of its irreducible components contains at least three points that are either marked or singular. We call $C^{\st}$ the \emph{stable model of $C$ corresponding to the marked disc $(D; x_1, \ldots, x_r)$}. See, e.g., \cite[\S 2.5.3]{Akeyrthesis} and \cite[\S 3.2]{2020arXiv200203719D} for more details.

\subsubsection{A coordinate system for a marked disc}
\label{seccordinatedisc}
A node (an ordinary double point) $e$ of the special fiber $\overline C$ corresponds to an annulus. Choose a marked point $x_{j,K}$ on the branch beyond $e$, viewed as a center in the coordinate $X$. The annulus is
\[
 \{P\in D\mid r_-(e)<\nu(X(P)-x_{j,K})<r_+(e)\}.
\]
We represent its radial interval by $[r_-(e),r_+(e)]$. A \emph{rational place} on $e$ is a rational number $r$ in this interval; it represents the Gauss valuation of the closed disc $\mathcal D[r,x_{j,K}]$, with the endpoints corresponding to the adjacent components. Increasing $r$ corresponds to shrinking the disc. The intervals are ordered along paths directed from the outer boundary toward smaller nested discs: we write $r<r'$ when $r$ precedes $r'$ on such a path. Places on different branches are not comparable.

\subsection{Refined Swan conductor of \texorpdfstring{$\mathbb{Z}/p^n$}{}-extensions}
Let $\chi \in \cohom^1(\mathbb{K}, \mathbb{Z}/p^{n})$ be a character, and let $\mathcal{D} \subset C^{\an}$ be a closed subdisc. Choose an affine coordinate $X$ in which $\mathcal D$ is the closed unit disc, after a finite extension of $K$ if necessary. Its canonical valuation is the Gauss valuation $\nu_0:=\nu_{0,0}$ defined in Section \ref{secdiscannuli}, and its residue field is $\kappa\cong k(x)$, where $x$ is the reduction of $X$. The degeneration of $\chi\lvert_{\mathcal{D}}$ is measured by the following \emph{refined Swan conductor} data:
\begin{enumerate}[label={\arabic*.}]
    \item The \emph{depth Swan conductor} $\delta({\chi\lvert_{\mathcal{D}}}) \in \mathbb{Q}_{\ge 0}$, which measures the separability of $\chi\lvert_{\mathcal{D}}$'s reduction. It is $0$ if and only if $\chi$ is unramified with respect to $\nu_0$,
    \item For $\delta(\chi\lvert_{\mathcal{D}})>0$, the \emph{differential Swan conductor} $\dsw(\chi\lvert_{\mathcal{D}}) \in \Omega^1_{k(x)}$, and
    \item For $\delta(\chi\lvert_{\mathcal{D}})=0$, a Witt vector $\underline{f}=(f^1, \ldots, f^n) \in W_n(k(x))$ representing the residual character, determined up to its Artin-Schreier-Witt class.
\end{enumerate}
We call the pair $(\delta(\chi\lvert_{\mathcal{D}}), \dsw(\chi\lvert_{\mathcal{D}}))$ when $\delta(\chi\lvert_{\mathcal{D}})>0$ (resp., $(0, \underline{f})$ when $\delta(\chi\lvert_{\mathcal{D}})=0$) the \emph{degeneration type} (resp., a \emph{reduction type}) of $\chi$ on $\mathcal{D}$. At depth zero, we also write $\underline f$ for the reduction type $(0,\underline f)$. 

\subsubsection{Variation of the depth on a disc}
\label{secboundaryswanrateofchange}
Fix a closed disc $\mathcal{D}$, which may be associated with the ring $R\{X\}$. Let $z$ be a $K$-point of $\mathcal{D}$ and $r \in \mathbb{Q}_{\ge 0}$. We denote by $\delta_{\chi}(r,z)$ (resp. $\delta_{\chi}(r)$) and $\omega_{\chi}(r,z)$ (resp. $\omega_{\chi}(r)$) the depth and the differential conductors of the restriction of $\chi$ to $\mathcal{D}[r,z]$ (resp. $\mathcal{D}[r]$). If $\overline{w}$ is a closed point of $\overline{\mathcal{D}}$, we denote by $U(\mathcal{D}, \overline{w}):=]\overline{w}[_{\mathcal{D}}$ the residue class of $\overline{w}$ on the affinoid $\mathcal{D}$.

When the point $z$ is fixed, we may regard $\delta_{\chi}(r,z)$ and $\omega_{\chi}(r,z)$ as functions of $r$. In fact, the depth extends to a continuous, piece-wise linear function:
\begin{equation*}
    \delta_{\chi}(\cdot, z): \mathbb{R}_{\ge 0} \rightarrow \mathbb{R}_{\ge 0},
\end{equation*}
whose kinks occur only at rational values of $r$, by \cite[Proposition 5.10]{MR3194815}.

The residue field of $\mathcal D[r,z]$ is the function field of its canonical reduction, whose projective completion is $\mathbb P^1_k$. A closed point $\overline w$ of this projective line determines a normalized discrete valuation $\ord_{\overline w}$ of that function field and a \emph{residue direction}. For a finite point $\overline w$, choose a lift $a\in\mathcal D[r,z](K)$ with that residue coordinate. The residue class $U(\mathcal D[r,z],\overline w)$ is the open disc $D[r,a]$; moving in this direction means shrinking discs inside this residue class. The point $\overline\infty$ determines the outward direction, toward discs containing $\mathcal D[r,z]$.

At positive depth, the differential conductor determines the directional slopes of the depth function, as the following proposition shows.

\begin{proposition}[{\cite[Corollary 5.11]{MR3194815}}]
\label{propdiffswanrateofchange}
If $r \in \mathbb{Q}_{\ge 0}$ and $\delta_{\chi}(r,z)>0$, then the right derivative of $\delta_{\chi}$ at $r$ is $-\ord_{\overline{0}}(\omega_{\chi}(r,z))-1$. If additionally $r>0$, its left derivative is $\ord_{\overline{\infty}}(\omega_{\chi}(r,z))+1$. In particular, $-\ord_{\overline{x}}(\omega_{\chi}(r,z))-1$ is the directional slope of the depth function in the residue direction $\overline{x}$.

If $r \in \mathbb{Q}_{\ge 0}$ and $\delta_{\chi}(r,z)=0$, then the right derivative of $\delta_{\chi}$ at $r$ is $\sw_{\overline{\chi}\lvert_{\mathcal{D}[r,z]}}(\overline{0})$, where $\overline{\chi}\lvert_{\mathcal{D}[r,z]}$ is the reduction of $\chi\lvert_{\mathcal{D}[r,z]}$ and $\sw_{\overline{\chi}\lvert_{\mathcal{D}[r,z]}}(\overline{x})$ is the usual Swan conductor of $\overline{\chi}\lvert_{\mathcal{D}[r,z]}$ with respect to the valuation $\ord_{\overline{x}}$ \cite[IV,\S 2]{MR554237}.
\end{proposition}

\begin{remark}
    Suppose $\overline{\chi} \in \cohom^1(k(x), \mathbb{Z}/p^n)$, which can be thought of as a $\mathbb{Z}/p^n$-cover of $\mathbb{P}^1_k$. Then, for each closed point $\overline{x} \in \mathbb{P}^1_k$, the integer $\sw_{\overline{\chi}}(\overline{x})$ is the maximal ramification break of $\overline{\chi}$ at $\overline{x}$, which can be calculated using Theorem \ref{theoremcaljumpirred}.
\end{remark}

The following proposition gives a criterion for good reduction.

\begin{proposition}[{(cf. \cite[Corollary 5.13]{MR3194815})}]
\label{propgood}
\begin{enumerate}[label=(\arabic*)]
    \item \label{propgooditem1} Let $\chi \in \cohom^1(\mathbb K,\mathbb Z/p^n)$ be an admissible character of order $p^n$ with $\delta_\chi(0)=0$ and residual character $\overline\chi$. Then
    \begin{equation}
        \label{eqngoodconductorswan}
        \iota_n:= \lvert \mathbb{B}(\chi) \cap U(\mathcal{D},\overline{0}) \rvert \ge \sw_{\overline{\chi}\lvert_{\mathcal{D}[0,0]}}(\overline{0})+1.
    \end{equation}
    \noindent Then $\chi$ has good reduction if and only if equality holds. We call $\iota_n$ the \textit{conductor} of the character $\chi$.
    \item \label{propgooditem2} Suppose $\chi$ has good reduction, and let $(m_1,\ldots,m_n)$ be the upper breaks of its residual character at $0$. Let $\chi_i:=\chi^{p^{n-i}}$. If $1 \le i \le n$, then the conductor of $\chi_i$ is
    \begin{equation}
        \label{eqnupperbreakconductor}
        \iota_i:=\lvert \mathbb{B}(\chi_i) \cap U(\mathcal{D},\overline{0}) \rvert=m_i+1
    \end{equation}
    For $2\leq i\leq n$, we have $\iota_i \ge p\iota_{i-1}-p+1$, and if $\iota_i \equiv 1 \pmod{p}$ then $\iota_i=p\iota_{i-1}-p+1$.
\end{enumerate}
\end{proposition}

For an admissible character on a disc $\mathcal D[r,z]\subseteq\mathcal D$ with $r>0$ and positive depth, the following result and Proposition \ref{propdiffswanrateofchange} bound the slope in each finite residue direction from above by the number of branch points in that residue class minus one. If the character has good reduction, this bound is attained.

\begin{proposition}[{(cf. \cite[Corollary 5.15]{MR3194815})}]
\label{propcriteriongoodnonroot}
Let $\chi$ be an admissible character of order $p^n$, and let $r\in\mathbb Q_{>0}$ and $z\in R$ satisfy $\mathcal D[r,z]\subseteq\mathcal D$ and $\delta_\chi(r,z)>0$. For every finite point $\overline w$ of the canonical reduction of $\mathcal D[r,z]$, we have
\begin{equation}
\label{eqnvanishingcyclenoninfty}
    \ord_{\overline{w}} (\omega_{\chi}(r,z)) \ge - |\mathbb{B}(\chi) \cap U(\mathcal{D}[r,z], \overline{w})|.
\end{equation}
Furthermore, equality holds if $\chi$ has good reduction.
\end{proposition}

\begin{corollary}
\label{cordiffgood}
    Suppose $\chi \in H^1(\mathbb{K}, \mathbb{Z}/p^n)$ has good reduction, and let $\mathcal{D}' \subsetneq \mathcal{D}$ be a proper subdisc such that the restricted character $\chi|_{\mathcal{D}'}$ is ramified. Then $\delta(\chi|_{\mathcal{D}'}) > 0$ and
    \begin{equation}\label{eqnomegagood}
        \omega(\chi|_{\mathcal{D}'}) = \frac{c \, dx}{\prod_{i=1}^l (x - a_i)^{m_i}},
    \end{equation}
    where $c \in k^\times$, each $a_i\in k$ is the residue coordinate of a finite pole direction in the canonical reduction of $\mathcal D'$, and $m_i := |\mathbb{B}(\chi) \cap U(\mathcal{D}', a_i)|$. In particular, let $(r_1,r_2)$ be an interval of positive depth with $0\leq r_1<r_2<1$, and assume that no branch point of $\chi$ lies in the open annulus $r_1<\nu(X)<r_2$. Then for every $r\in(r_1,r_2)\cap\mathbb Q$ we have
    \begin{equation*}
        \delta_{\chi}(r) = \delta_{\chi}(r_2) - (|\mathbb{B}(\chi) \cap \mathcal{D}[r_2]| - 1)(r_2 - r).
    \end{equation*}
\end{corollary}

\begin{proof}
    The claim that $\ord_{a_i}(\omega(\chi|_{\mathcal{D}'})) = -m_i$ follows from \eqref{eqnvanishingcyclenoninfty}. Without loss of generality, assume $\mathcal{D}' = \mathcal{D}[r',z]$. As the depth function is piece-wise linear, there exists a rational $r < r'$ such that $\delta_{\chi}(\cdot, z)$ is linear on $(r, r')$ and $\mathbb{B}(\chi) \cap \mathcal{D}' = \mathbb{B}(\chi) \cap \mathcal{D}[r, z]$. It then follows from Proposition \ref{propcriteriongoodnonroot} and Proposition \ref{propdiffswanrateofchange} that the slope of $\delta_{\chi}(\cdot, z)$ on $(r, r')$ is $|\mathbb{B}(\chi) \cap \mathcal{D}'| - 1 = \sum_{i=1}^l m_i - 1$. Again by Proposition \ref{propdiffswanrateofchange}, $\ord_{\infty}(\omega(\chi|_{\mathcal{D}'})) = \sum_{i=1}^l m_i - 2$. Hence the differential has the form \eqref{eqnomegagood}. On $(r_1,r_2)$ the slope is $|\mathbb B(\chi)\cap\mathcal D[r_2]|-1$, since the closed disc includes branch points of valuation exactly $r_2$. Integrating this constant slope and using continuity gives the second assertion.
\end{proof}

\subsubsection{Conditions on the refined Swan conductors of cyclic covers}
Fix the Gauss valuation of a closed disc. Let $\chi_n \in \cohom^1(\mathbb K,\mathbb Z/p^n)$ have order $p^n$, with $n\geq1$, and suppose that the residue extension at this valuation is purely inseparable of degree $p^n$. For $i=1, \ldots, n$, we set
$$ \chi_i:=\chi_n^{p^{n-i}}, \hspace{5mm} \delta_i:=\delta_{\chi_i}, \hspace{5mm} \omega_i:=\omega_{\chi_i}.  $$
The tuple $(\delta_i, \omega_i)_{i=1, \ldots, n}$ is called the \emph{ramification datum} associated to $\chi_n$. For $1 \le j \le n$, the pair $(\delta_j, \omega_j)$ is called $\chi_n$'s \emph{$j$-th ramification datum}.

The next theorem describes the possible depth and differential conductors of $\chi_n$ when those of $\chi_{n-1}$ are prescribed.

We write $\mathcal C$ for the \emph{Cartier operator} on rational differentials over $k$. In a local parameter $u$, it acts on a Laurent expansion by
\[
 \mathcal C\!\left(\sum_{j\in\mathbb Z}b_j u^j\,du\right)
 =\sum_{q\in\mathbb Z}b_{pq+p-1}^{1/p}u^q\,du.
\]
The $p$th roots of the coefficients are unique because $k$ is perfect. A rational differential is called \emph{exact} if it has the form $dg$ for $g\in k(x)$, and \emph{logarithmic} if it has the form $dg/g$ for $g\in k(x)^\times$.

\begin{theorem}[{(cf. Theorem 4.3 \cite{MR3167623})}]
\label{theoremCartierprediction}
Let $\chi \in \cohom^1(\mathbb{K}, \mathbb{Z}/p^{n})$ be a cyclic character of order $p^n$, whose associated residue extension is purely inseparable of degree $p^n$. Let $(\delta_i, \omega_i)_i$ denote the ramification data associated with $\chi$. Then the following holds.
\begin{enumerate}[label=(\arabic*)]
    \item $0 < \delta_1 \le \frac{p}{p-1}$. Moreover, we have
    \begin{enumerate}
        \item $\delta_1=\frac{p}{p-1} \iff \mathcal{C}(\omega_1)=\omega_1, $
        \item $\delta_1< \frac{p}{p-1} \iff \mathcal{C}(\omega_1)=0$.
    \end{enumerate}
    \item Suppose $i>1$. If $\delta_{i-1}>\frac{1}{p-1}$ then we have
    \[\delta_i=\delta_{i-1}+1, \enspace \omega_i=-\omega_{i-1}.\]
    \item Suppose $i>1$. If $\delta_{i-1} \le \frac{1}{p-1}$ then
    \[p\delta_{i-1} \le \delta_i \le \frac{p}{p-1}.\]
    Moreover, we have
    \begin{enumerate}
        \item $p\delta_{i-1}=\delta_i<\frac{p}{p-1} \Rightarrow \mathcal{C}(\omega_i)=\omega_{i-1}$
        \item $p\delta_{i-1} < \delta_i < \frac{p}{p-1} \Rightarrow \mathcal{C}(\omega_i)=0$
        \item $p\delta_{i-1} < \delta_i =\frac{p}{p-1} \Rightarrow \mathcal{C}(\omega_i)=\omega_i$
        \item $p\delta_{i-1} = \delta_i =\frac{p}{p-1} \Rightarrow \mathcal{C}(\omega_i)=\omega_i+\omega_{i-1}$.
    \end{enumerate}
\end{enumerate}
\end{theorem}

\section{Hurwitz tree}
\label{secHurwitz}
Let $R$ be a complete discrete valuation ring of mixed characteristic $(0,p)$ with fraction field $K$ and algebraically closed residue field $k$. We associate Hurwitz trees to cyclic covers of $C=\mathbb P^1_K$ that are {\'e}tale outside an open disc $D\subset C^{\an}$, using semistable $R$-models of $C$. These trees give an obstruction to refined local lifting, analogous to the lifting obstruction in \cite{MR2534115}.
 
\subsection{Orientation on a rooted tree}
A \emph{rooted tree} $T=(V,E)$ has a distinguished vertex $v_0$, called the \emph{root}, incident to a unique edge $e_0$, called the \emph{trunk}. Orient every edge away from the root, and write $s(e)$ and $t(e)$ for its source and target. If $v=s(e)$ and $v'=t(e)$, we call $v'$ a \emph{successor} of $v$ and write $v\to v'$. We write $v\leq v'$ if there is a directed path from $v$ to $v'$, and $v\prec v'$ if the path has positive length. The target $v_1=t(e_0)$ is called the \emph{first vertex}.

The root $v_0$ is the unique minimal vertex. A maximal component vertex is called \emph{terminal}. We keep component vertices separate from marked leaves: $B$ is a finite nonempty set indexing marked points, represented by legs attached to component vertices. These legs are not metric edges in $E$. If $v(b)$ is the component vertex carrying $b\in B$, put
\[
 B_v:=\{b\in B\mid v\leq v(b)\}.
\]
Thus $B_v$ consists of the marked leaves on or beyond $v$.

\subsection{Hurwitz tree and the deformation problem}

\begin{definition}
\label{defdecorated}
A \emph{decorated tree} is given by the following data
\begin{itemize}
    \item a semi-stable curve $\overline{C}$ over $k$ of genus $0$,
    \item a family $(x_b)_{b\in B}$ of pairwise distinct smooth $k$-rational points of $\overline{C}$, indexed by a finite nonempty set $B$,
    \item a distinguished smooth $k$-rational point $x_0 \in \overline{C}$, distinct from all the marked points $x_b$.
\end{itemize}
We allow subdivisions by unmarked rational components with two special points (nodes or $x_0$). We require $\overline C$ to be stably marked by $((x_b),x_0)$ after these subdivision components are contracted. Their vertices are \emph{inessential}; a nonroot vertex is \emph{essential} if its component survives stabilization.
\end{definition}

The \emph{combinatorial tree} $T=(V,E)$ of a decorated tree $\overline C$ is obtained from its dual graph by adjoining a root $v_0$ and a trunk $e_0$ joining it to the vertex of the component containing $x_0$. Write $\overline C_v$ for the component corresponding to $v\ne v_0$, and $z_e$ for the node corresponding to $e\ne e_0$; set $z_{e_0}=x_0$. See \cite[Example 3.1]{2020arXiv200203719D} for an example.

Since $\overline C$ has genus zero, its components are rational and $T$ is a tree. For $v\ne v_0$, obtain $\overline U_v$ from $\overline C_v$ by removing the nodes and marked points, and also $x_0$ when $v=t(e_0)$.

The following definition applies in both mixed and equal characteristic.

\begin{definition}
\label{defnhurwitztree}
A \emph{differential Hurwitz tree} $\mathcal{T}$ of type $G = \mathbb{Z}/p^n$ consists of the following data:
\begin{itemize}
    \item A decorated tree $\overline{C}=(\overline{C},(x_b),x_0)$ with underlying rooted tree $T=(V,E)$.
    \item For every $v \in V$, a rational $0 \le \delta_{\mathcal{T}}(v)$, called the \emph{depth} of $\mathcal{T}$ at $v$. 
    \item For each $v \in V$ such that $\delta_{\mathcal{T}}(v) > 0$, a differential form $\omega_{\mathcal{T}}(v) \in \Omega^1_{k(x)}$, called the \emph{differential conductor} at $v$.
    \item For each $v \in V$ and each marked leaf $b\in B$, subgroups $G_{\mathcal{T}}(v)$ and $G_{\mathcal{T}}(b)$ of $G$, called their \emph{monodromy groups}.
    \item For every $e \in E$, a positive rational number $\epsilon_e$, called the \emph{thickness} of $e$.
    \item For every $e \in E$, a positive integer $d_{\mathcal{T}}(e)$, called the \emph{slope} on $e$. At a vertex $v$, this is an \emph{incoming slope} if $t(e)=v$ and an \emph{outgoing slope} if $s(e)=v$.
    \item For every $b \in B$, a positive integer $h_{\mathcal{T}}(b)$, called the \emph{conductor} at $b$.
    \item For $v_0$ with $\delta_{\mathcal{T}}(v_0)=0$, a \emph{reduced} length-$n$ Witt vector $\underline{f}:=(f^1, \ldots,f^n) \in W_n(k(x))$ with only the pole at $0$, called the \emph{reduction type} of the tree. Each $f^i$ is the \emph{$i$-th coordinate of the reduction type}.
\end{itemize}
These data satisfy the following conditions:
\begin{enumerate}[label=(\text{H}\arabic*)]
    \item \label{c1Hurwitz} Every nonroot vertex has positive depth.
    \item \label{c2Hurwitz} For each $v \in V \setminus \{v_0\}$, the differential form $\omega_{\mathcal{T}}(v)$ does not have zeros or poles on $\overline{U}_v \subsetneq \overline{C}_v$.
    \item \label{c3Hurwitz} For every edge $e \in E \setminus \{e_0\}$, we have the equality 
    \[ -\ord_{z_e} \omega_{\mathcal{T}}(t(e))-1 = \ord_{z_e}\omega_{\mathcal{T}}(s(e))+1. \]
    \item \label{c4Hurwitz} For every edge $e\in E$, we have
    \[ d_{\mathcal{T}}(e) = \ord_{z_e} \omega_{\mathcal{T}}(t(e))+1 \underset{s(e) \neq v_0}{\stackrel{\text{\ref{c3Hurwitz}}}{=}} -\ord_{z_e} \omega_{\mathcal{T}}(s(e))-1. \]
    At a positive-depth root we require $d_{\mathcal T}(e_0)=-\ord_0\omega_{\mathcal T}(v_0)-1$.
    \item \label{c5Hurwitz} For every edge $e\in E$, we have
    \[\delta_{\mathcal{T}}(s(e))+\epsilon_e  d_{\mathcal{T}}(e) = \delta_{\mathcal{T}}(t(e)). \]
    \item \label{c6Hurwitz} For $b\in B$, let $\overline{C}_v$ be the component containing the point $x_b$. Then the differential $\omega_{\mathcal{T}}(v)$ has a pole at $x_b$ of order $h_{\mathcal{T}}(b)$.
    \item \label{c6rootHurwitz} At a depth-zero root, let $u_i$ be the final upper jump of the cyclic character represented by $(f^1,\ldots,f^i)$, and set $u_0=-1$. We require $d_{\mathcal T}(e_0)=u_n$ and
    \[
    \sum_{\substack{b\in B\\G_{\mathcal T}(b)\cong\mathbb Z/p^{n-i+1}}}
    h_{\mathcal T}(b)=u_i-u_{i-1}\qquad(1\leq i\leq n).
    \]
    \item \label{c7Hurwitz} For each $v$, we have
    \[ G_{\mathcal{T}}(v') \leq G_{\mathcal{T}}(v), \]
    for every successor vertex $v'$ of $v$. Moreover, we have
    \[
    \sum_{v\to v'}[G_{\mathcal T}(v):G_{\mathcal T}(v')]
    +\sum_{x_b\in\overline C_v}[G_{\mathcal T}(v):G_{\mathcal T}(b)]>1
    \]
    at every essential nonroot vertex, where each marked-leaf group is a subgroup of $G_{\mathcal T}(v)$. At an unmarked subdivision vertex there is one successor of the same group, and the sum is $1$. At the root there is exactly one successor $v'$ and $G_{\mathcal T}(v_0)=G_{\mathcal T}(v')=G$.
\end{enumerate}
\end{definition}

For each $v \in V \setminus \{v_0\}$, we call $(\delta_{\mathcal{T}}(v), \omega_{\mathcal{T}}(v))$ the \emph{degeneration type} at $v$. When $\delta_{\mathcal{T}}(v_0) > 0$, the pair $(\delta_{\mathcal{T}}(v_0), \omega_{\mathcal{T}}(v_0))$ is called the degeneration type of $\mathcal{T}$. Set $d_{\mathcal T}:=d_{\mathcal T}(e_0)$.
The positive integer $\mathfrak{C}_{\mathcal{T}} := d_{\mathcal{T}} + 1$ is called the \emph{conductor} of the Hurwitz tree. If $v = t(e)$, we define $\mathfrak{C}_{\mathcal{T}}(v) := d_{\mathcal{T}}(e) + 1$. The rational number $\delta_{\mathcal{T}} := \delta_{\mathcal{T}}(v_0)$ is called the \emph{depth} of $\mathcal{T}$. We say a tree $\mathcal{T}$ is \emph{étale} if $\delta_{\mathcal{T}} = 0$; otherwise, it is described as \emph{radical}. When $b \in B$ and $G_{\mathcal T}(b) = \mathbb{Z}/p^i$, $b$ is termed a leaf of \emph{index} $i$.

\begin{remark}
\label{remarkslopesumconductors}
For each $e\in E$, conditions \ref{c2Hurwitz}, \ref{c3Hurwitz}, \ref{c4Hurwitz}, and \ref{c6Hurwitz} give
\begin{equation}
\label{eqnslopeconductor}
    d_{\mathcal T}(e)=\sum_{b\in B_{t(e)}}h_{\mathcal T}(b)-1>0.
\end{equation}
In particular, $\mathfrak C_{\mathcal T}=\sum_{b\in B}h_{\mathcal T}(b)$.
\end{remark}

\subsection{Hurwitz trees arise from cyclic covers}
\label{seccovertotree}
Fix $G=\mathbb Z/p^n$ and take $R$ to be a finite extension of $W(k)[\zeta_{p^n}]$. Let $\chi\in H^1(\mathbb K,G)$, with $\mathbb K=K(X)$, be an admissible character of order $p^n$. Let $A$ be the normalization of $R\{X\}$ in the finite $\operatorname{Frac}(R\{X\})$-algebra induced by $\chi$. This gives
\[
\Phi:\operatorname{Spec}A\longrightarrow\operatorname{Spec}R\{X\}.
\]
Its analytic generic fibre is the cover of $\mathcal D$ associated with $\chi$.
For trees arising from covers, fix the following choice of Wewers's normalization \cite[\S2.4]{MR3167623}: choose compatible rational powers of $\Lambda=\zeta_p-1$ and put $\pi_r=\Lambda^{(p-1)r}$, so that $\nu(\pi_r)=r$ and $\pi_{r+s}=\pi_r\pi_s$. On $\mathcal D[r,z]$, use the residue coordinate $x=\overline{(X-z)/\pi_r}$.

Use the semistable model of the marked disc from \S\ref{secsemistablemodel}. Its exceptional curve, marked by the specialized branch points and $x_0$, gives the decorated tree. For $v\ne v_0$, let $U_v$ be the affinoid subdomain with reduction $\overline U_v$. The construction of \cite[\S4.2]{2020arXiv200203719D} assigns the data in Table \ref{tab:covertotree}; the Hurwitz-tree conditions will follow from good reduction.

\begin{remark}
    We have not assigned a single slope to each edge: when $\chi$ has bad reduction, the depth need not be linear on the corresponding annulus.
\end{remark}

\begin{table}[ht]
    \centering
\begin{tabular}{ |p{5.4cm}|p{9.6cm}|  }
\hline
Data on $\mathcal{T}$ & Degeneration data of the cover \\
\hline
\hline
The decorated tree & The exceptional curve $\overline C$ over the marked disc, with its marked points and $x_0$. \\
\hline
The depth of $v$ & The depth of $\Phi$ on $U_v$ for $v\ne v_0$, and $\delta_\chi(0)$ at $v_0$   \\
\hline
The differential conductor at a positive-depth $v$ & The differential conductor of $\Phi$ on $U_v$ for $v\ne v_0$, and $\omega_\chi(0)$ at $v_0$ \\
\hline
The thickness of an edge $e$ & The thickness of the corresponding annulus, measured in the radius coordinate for $\nu(p)=1$ \\
\hline
The conductor at a leaf $b$ & The local Artin conductor of the character at the corresponding generic branch point; it is $1$ in characteristic zero \\
\hline
The reduction type at a depth-zero $v_0$ & The reduced Witt vector representing the residual character of $\Phi$ \\
\hline
The monodromy groups & $G_{\mathcal T}(b)=I_b$, where $I_b$ is the inertia group at the corresponding branch point; $G_{\mathcal T}(v)$ is the largest $I_b$ for $b\in B_v$ \\
\hline
\end{tabular}
    \caption{Assigning a Hurwitz tree to an admissible cover}
    \label{tab:covertotree}
\end{table}

Classically, the existence of a Hurwitz tree of type $\mathbb{Z}/p^n$ provides a necessary condition for the lifting problem.

\begin{proposition}
\label{propobstructionlifting}
    Let $\overline{\chi} \in H^1(k(x), \mathbb{Z}/p^n)$ be the character of the one-point cover associated with the local extension in \S\ref{secsetup}. If $\overline{\chi}$ lifts to characteristic $0$, then there exists a differential étale Hurwitz tree $\mathcal{T}$ of type $\mathbb{Z}/p^n$ such that
    \begin{equation*}
        \mathfrak{C}_{\mathcal{T}} = \sw_{\overline{\chi}}(\overline{0})+1 \quad \text{and} \quad \delta_{\mathcal{T}} = 0.
    \end{equation*}
\end{proposition}

\begin{proof}
    Let $\mathcal T$ be the tree constructed from a lift $\chi$ of $\overline\chi$. For each edge $e$, we set $d_{\mathcal{T}}(e)$ to be $\ord_{z_e}(\omega_{\mathcal{T}}(t(e)))+1$, which immediately satisfies \ref{c4Hurwitz}. Good reduction and Corollary \ref{cordiffgood} give conditions \ref{c2Hurwitz} and \ref{c3Hurwitz}. The remaining checks are positivity of nonroot depths, the depth--slope relation, the marked-pole orders, and the monodromy conditions, as in \cite[Theorem 3.9]{MR2534115}. The root has depth zero, and the conductor counts in \ref{c6rootHurwitz} follow by applying Proposition \ref{propgood} to the quotient covers. Proposition \ref{propgood} also gives $\mathfrak C_{\mathcal T}=\sw_{\overline\chi}(\overline 0)+1$.
\end{proof}

\begin{remark}
    The cyclic lifting theorem of Obus--Wewers and Pop \cite{MR3194815,MR3194816} implies that every $\overline\chi$ admits a tree as in Proposition \ref{propobstructionlifting}.
\end{remark}

For order $p$, realization requires the mixed-characteristic admissibility conditions, in addition to Definition \ref{defnhurwitztree}. In our normalization these include $0<\delta(v)\leq p/(p-1)$ at nonroot vertices, exact differentials below $p/(p-1)$, logarithmic differentials at $p/(p-1)$, and marked leaves of conductor $1$ attached at that maximal depth. We also require the root compatibility rule of Theorem \ref{theoremcompatibilitydiff}.

\begin{theorem}[(cf. {\cite[Theorem 3.5]{2000math.....11098H}})]
\label{thminversedegreep}
Let $\mathcal T$ be an \'etale tree of type $\mathbb Z/p$ satisfying the mixed-characteristic admissibility conditions of \cite[\S2 and Theorem 3.5]{2000math.....11098H}, with $\mathfrak C_{\mathcal T}=m$. After a finite extension of $K$, it is realized by a $\mathbb Z/p$-lift of conductor $m$.
\end{theorem}

For a cyclic character $\chi$ of order $p^n$, its \emph{quotient tower} consists of the characters $\chi_i=p^{n-i}\chi$ of order $p^i$, for $1\leq i\leq n$. Write $\delta_i(v)$ and $\omega_i(v)$ for their depth and differential data, with the latter defined at positive depth. A differential $\eta$ is \emph{Cartier-nilpotent} if $\mathcal C^s(\eta)=0$ for some integer $s\geq1$.

Relative to two successive levels $i-1$ and $i$, with $i>1$, a vertex $v$ is \emph{critical} if $\delta_{i-1}(v)=1/(p-1)$, and \emph{subcritical} if $\delta_{i-1}(v)<1/(p-1)$. Theorem \ref{theoremCartierprediction} forces $\delta_i(v)=p/(p-1)$ at a critical vertex and gives $\mathcal C(\omega_i(v))=\omega_i(v)+\omega_{i-1}(v)$. We call $\omega_i(v)$ its \emph{critical differential}. We use the same terminology for corresponding vertices in an extension of differential Hurwitz trees.

A Hurwitz tree arising from a mixed-characteristic cover satisfies the following additional conditions.

\begin{proposition}
\label{propmixedtree}
Let $\chi\in\cohom^1(\mathbb K,\mathbb Z/p^n)$ be an admissible character of order $p^n$ over the mixed-characteristic field $K$, whose degeneration data define a differential Hurwitz tree $\mathcal T$ as in \S\ref{seccovertotree}. Then:
\begin{enumerate}[label=(\arabic*)]
\item\label{propmixedtree1} If marked leaves are attached directly to $v$, and $i(v)$ is their maximal index, then $\delta_{\mathcal T}(v)=p/(p-1)+i(v)-1$.
\item\label{propmixedtree2} If $v$ is terminal, then $\omega_{\mathcal T}(v)$ is logarithmic.
\item\label{propmixedtree3} Every marked leaf has conductor $1$.
\end{enumerate}
\end{proposition}
\begin{proof}
The depth formula is \cite[Lemma 3.10(iii)]{MR2534115}. All leaves attached directly to a fixed vertex have the same index: otherwise, after passing down the quotient tower until the smallest index is $1$, the larger index forces the lower depth to be above $1/(p-1)$. Theorem \ref{theoremCartierprediction} would then identify the upper differential with the negative of the lower one, although the upper differential has a pole in the index-$1$ direction and the lower differential is regular there.

At a terminal vertex whose attached leaves have index $1$, the local monodromy is of order $p$ and the depth is $p/(p-1)$, so the differential is logarithmic. At a terminal vertex whose attached leaves have higher index, quotienting lowers every leaf index by one, and Theorem \ref{theoremCartierprediction} gives $\omega=-\omega_{\rm lower}$. Induction proves \ref{propmixedtree2}.

For \ref{propmixedtree3}, first take a leaf of index $1$. Here the lower depth and differential refer to the immediately preceding level of the quotient tower. If its vertex is terminal, or the lower depth is strictly below $1/(p-1)$, the upper differential is logarithmic and the pole is simple.

In the remaining case the lower differential $\omega'$ is Cartier-nilpotent and the critical relation is $\mathcal C(\omega)=\omega+\omega'$. Choose $s$ with $\mathcal C^s(\omega')=0$ and put $\eta=-\sum_{j=0}^{s-1}\mathcal C^j(\omega')$. Then $\mathcal C(\omega-\eta)=\omega-\eta$. By the characterization of Cartier-fixed differentials as logarithmic, $\omega-\eta=dg/g$ for some $g\in k(x)^\times$. The lower character is unramified in the direction corresponding to this index-$1$ leaf, and $\omega'$ is regular there. Cartier preserves regularity, so $\eta$ is also regular there; the pole comes from $dg/g$ and is simple.

For a leaf of higher index, quotienting and the identity $\omega=-\omega_{\rm lower}$ reduce to the preceding case.
\end{proof}

\begin{remark}
A critical vertex may carry new order-$p$ leaves together with old branches. Its differential satisfies $\mathcal C(\omega_n)=\omega_n+\omega_{n-1}$ and need not be logarithmic.
\end{remark}

The proof of the following proposition is similar to \cite[Proposition 4.15]{2020arXiv201013614D}.

\begin{proposition}
\label{propinverse}
In Question \ref{questionrefinedHKG}, let $\Phi_n$ extend $\Phi_{n-1}$. If its differential Hurwitz tree $\mathcal T_n$ is \'etale with the prescribed reduction type $(g^1,\ldots,g^n)$, then $\Phi_n$ solves the refined local lifting problem.
\end{proposition}

\begin{proof}
The root depth is zero and the prescribed root type identifies the residual character with $(g^1,\ldots,g^n)$. Let $u_n$ be its final upper jump. By Proposition \ref{propmixedtree}\ref{propmixedtree3}, all marked leaves have conductor $1$. Summing \ref{c6rootHurwitz} gives
\[
 |\mathbb B(\Phi_n)|=\sum_{b\in B}h_{\mathcal T_n}(b)=u_n+1.
\]
The good-reduction criterion in Proposition \ref{propgood}\ref{propgooditem1} therefore shows that $\Phi_n$ has good reduction. Its reduction is the prescribed one-point cover, and its quotient is the given $\Phi_{n-1}$.
\end{proof}

\subsection{The compatibility of the differential conductors}
\label{seccompatibility}

Fix residue coordinates on the nonroot components with the incoming point at infinity, and retain the prescribed root coordinate. For trees arising from covers, use the normalization fixed in \S\ref{seccovertotree}. All coefficient comparisons use these choices.

\begin{definition}
\label{defnconstantcoeffepart}
Let $\mathcal T$ be a differential Hurwitz tree. At a nonroot component $v$, put $S_v:=\overline C_v\setminus(\overline U_v\cup\{\infty\})$, the set of outgoing nodes and marked points in the fixed coordinate $x$. For $q\in S_v$, let $a_q\in k$ be its residue coordinate; for an outgoing edge $e$, write $a_e:=a_{z_e}$. Write
\[
 \omega_{\mathcal T}(v)=\frac{c_v\,dx}{\prod_{q\in S_v}(x-a_q)^{h_q}},\qquad
 h_q=\begin{cases}
 d_{\mathcal T}(e)+1,&q=z_e\text{ for an outgoing edge }e,\\
 h_{\mathcal T}(b),&q=x_b\text{ for a marked leaf }b.
 \end{cases}
\]
We call $c_v\in k^\times$ the \emph{constant coefficient at $v$}, also called the \emph{infinity coefficient}. If the monic denominator has degree $q$, then
\[
 \omega_{\mathcal T}(v)=\bigl(c_vx^{-q}+O(x^{-q-1})\bigr)\,dx.
\]
Thus in the parameter $u=1/x$, the leading term is $-c_vu^{q-2}\,du$.

The \emph{principal part} at a point is the sum of the negative-power terms in the Laurent expansion of the differential in a local coordinate. For an outgoing edge $e$, its \emph{$e$-coefficient} is the coefficient of $dx/(x-a_e)^{d_{\mathcal T}(e)+1}$ in the principal part of $\omega_{\mathcal T}(v)$ at $a_e$.

At a positive-depth root with prescribed principal part $\sum_{j=2}^{N+1}a_j\,dx/x^j$, $a_{N+1}\ne0$, the root coefficient, also called the $e_0$-coefficient, is $a_{N+1}$.

At a depth-zero root, use the coefficient rule of Theorem \ref{theoremcompatibilitydiff}. These conventions apply to abstract trees as well as to trees arising from covers.
\end{definition}

\begin{definition}
\label{defnconstantpartonedge}
If $\mathcal T$ arises from a good-reduction character $\chi$, let $r$ be a rational interior place of an edge $e$. In the normalized residue coordinate at $r$, let $a_{e,r}\in k$ denote the coordinate of the unique finite direction along $e$ toward $t(e)$. Proposition \ref{propcriteriongoodnonroot} gives a differential conductor of the form
\[ \omega_{\chi}(r)=\sum_{1\leq j\leq l} \frac{c_j dx}{(x-a_{e,r})^j} \]
for some $l \in \mathbb{N}$ with $c_l \neq 0$. We call $c_l$ the \emph{coefficient at $r$}.
\end{definition}

\begin{theorem}
\label{theoremcompatibilitydiff}
Suppose $\mathcal{T}$ is a tree that arises from a cyclic cover $\chi$ with good reduction, equipped with the residue coordinates fixed at the beginning of \S\ref{seccompatibility}. Let $e$ be an edge in $\mathcal{T}$. If $\delta_{\mathcal{T}}(s(e))>0$, then the following hold.
\begin{enumerate}[label=(\arabic*)]
    \item At every rational interior place $r$ of $e$, the coefficient at $r$ is equal to the constant coefficient at $t(e)$.
    \item The constant coefficient at $t(e)$ is equal to the $e$-coefficient at $s(e)$.
\end{enumerate}
For an {\'e}tale tree of order $p^n$, $n>1$, let $m_{n-1}$ be the final upper jump of the lower root type $(f^1,\ldots,f^{n-1})$. If $f^n=0$, or $\deg_{x^{-1}}f^n<pm_{n-1}$, no additional root coefficient is imposed. If $f^n=d_lx^{-l}+\cdots$ with $l>pm_{n-1}$, the differential at $t(e_0)$ has order $l-1$ at infinity and infinity coefficient $-ld_l$. For $n=1$, if $f^1=d_lx^{-l}+\cdots$, the differential at $t(e_0)$ has order $l-1$ at infinity and infinity coefficient $-ld_l$.
\end{theorem}

The proof for equal characteristic is given in \cite[Theorem 4.19]{2020arXiv201013614D}. For mixed characteristic, we use the edge- and root-compatibility results proved in \cite{DangObusRefinedLocalLiftingP2}.

\begin{definition}
\label{defncompatibility}
If $\delta_{\mathcal T}(s(e))>0$, we say the differential conductors at $s(e)$ and $t(e)$ are \emph{compatible along $e$} if the constant coefficient at $t(e)$ equals the $e$-coefficient at $s(e)$. At a depth-zero root, compatibility along the trunk means the root rule in Theorem \ref{theoremcompatibilitydiff}. For $v\leq v'$, we say $v'$ is compatible with $v$ if compatibility holds along every edge of the path from $v$ to $v'$.
\end{definition}

\subsubsection{Enrichment along edges}
\label{secsubtree}
\begin{definition}
\label{defsumcondsprec}
An \emph{enriched} tree carries interpolated edge data. Parameterize $e$ by distance $\rho\in[0,\epsilon_e]$ from $s(e)$; in the radial coordinate of \S\ref{seccordinatedisc},
\[
 \rho=r-r_-(e).
\]
At each rational distance, set
\[
 \delta_{\mathcal T}(\rho,e)=\delta_{\mathcal T}(s(e))+d_{\mathcal T}(e)\rho.
\]
At an interior place, use the normalized residue coordinate with the outgoing direction at $x=0$, and set
\[
 \omega_{\mathcal T}(\rho,e)=\frac{c_e\,dx}{x^{d_{\mathcal T}(e)+1}},
\]
where $c_e$ is the infinity coefficient at $t(e)$; the endpoint differentials remain unchanged.

For successive levels $i-1,i$ with $i>1$, let $e\in E(\mathcal T_{i-1})$ and $\rho\in[0,\epsilon_e]\cap\mathbb Q$. The place $(\rho,e)$ is \emph{critical} if $\delta_{\mathcal T_{i-1}}(\rho,e)=1/(p-1)$ and \emph{subcritical} if $\delta_{\mathcal T_{i-1}}(\rho,e)<1/(p-1)$.

\end{definition}

\subsection{The extension of a Hurwitz tree}

\begin{definition}
\label{defnextendhurwitz}
Suppose we are given a $\mathbb{Z}/p^{n-1}$-tree $\mathcal{T}_{n-1}$ and a $\mathbb{Z}/p^{n}$-tree $\mathcal{T}_{n}$. We say that $\mathcal{T}_n$ \emph{extends} $\mathcal{T}_{n-1}$, denoted
by $\mathcal{T}_{n-1} \prec \mathcal{T}_n$, if the following conditions hold.
\begin{enumerate}[label=(\arabic*)]
    \item After subdividing edges of the two trees as needed by inserting unmarked rational components with two special points, the lower decorated tree embeds in the upper one, preserving the root, incidence relations, marked points, and component coordinates. Vertices and leaves identified by this embedding are \emph{corresponding}; additional branches and leaves may be attached. The upper segments over each original lower edge $e$ have total thickness $\epsilon_e$.
    \item At each corresponding positive-depth vertex of $\mathcal T_{n-1}$, the lower and upper depth and differential conductors satisfy Theorem \ref{theoremCartierprediction}. The depth-zero root is governed by the Witt-vector condition below.
    \item At each corresponding vertex or marked leaf with lower group $\mathbb Z/p^i$, the upper group is $\mathbb Z/p^{i+1}$.
    \item Every vertex or marked leaf of $\mathcal T_n$ without a counterpart in the subdivided lower tree has group $\mathbb Z/p$. A vertex inserted on a lower edge has a corresponding lower subdivision vertex.
    \item Suppose $\delta_{\mathcal{T}_n}(v_0)=\delta_{\mathcal{T}_{n-1}}(v_0)=0$. Then the reduction type of $\mathcal{T}_n$ is a length-$n$ Witt vector $(f^1, \ldots, f^{n-1}, f^n)$ such that $(f^1, \ldots, f^{n-1})$ is the reduction type of $\mathcal{T}_{n-1}$.
\end{enumerate}
We say $\mathcal{T}_n$ extends a $\mathbb{Z}/p^i$-tree $\mathcal{T}_i$ if there exists a sequence of extensions $\mathcal{T}_i \prec \mathcal{T}_{i+1} \prec \ldots \prec \mathcal{T}_{n-1} \prec \mathcal{T}_n$.
\end{definition}

Applying the construction in Section \ref{seccovertotree} to a tower of covers gives the following necessary condition for lifting.

\begin{proposition}
\label{propcriterionhurwitzlift}
    In Question \ref{questionrefinedHKG}, let $\mathcal T_{n-1}$ be the {\'e}tale Hurwitz tree of $\Phi_{n-1}$. A solution $\Phi_n$ gives a Hurwitz tree $\mathcal T_n$ extending $\mathcal T_{n-1}$ with root reduction type equal to that of $\phi_n$. 
\end{proposition}

We call the tree-extension requirement in Proposition \ref{propcriterionhurwitzlift} the \emph{differential Hurwitz tree obstruction} for the refined lifting problem.

\begin{example}
\label{exunique}
Suppose $p=2$ and $\overline{\chi}_2 \in \cohom^1(k(x), \mathbb{Z}/4)$ is represented by $\left( \frac{1}{x^3}, 0 \right)$. Assume $\chi_1$, a lift of $\overline{\chi}_1 = \overline{\chi}_2^2$, gives rise to the Hurwitz tree $\mathcal{T}_1$ shown in Figure \ref{fig:uniqueextension} with degeneration data in Table \ref{tab:degenerationdataexunique}. The tree $\mathcal T_2$ depicted in Figure \ref{fig:uniqueextension} gives an extension of $\mathcal T_1$ with this root type.
The $v_2$ and $v_3$ entries for $\mathcal T_1$ are unmarked subdivision places corresponding to the vertices of $\mathcal T_2$.
\begin{figure}[ht]
    \centering
\begin{tikzpicture}[
    x=1cm,y=1cm,line cap=round,line join=round,
    old branch/.style={draw=blue!70!cyan,line width=0.8pt},
    new branch/.style={draw=violet,line width=0.8pt},
    every node/.style={font=\small,inner sep=2pt}
]
\foreach \panel/\treename in {0/{\mathcal T_1},1/{\mathcal T_2}} {
  \begin{scope}[xshift=4.2*\panel cm]
    \coordinate (root) at (0,0);
    \coordinate (fork) at (0,1.2);
    \coordinate (left tip) at (-1.1,2.3);
    \coordinate (right tip) at (0.8,2.3);
    \draw[old branch] (root)--(fork)--(left tip)
      (fork)--(right tip);
    \draw[old branch] (left tip)--++(-0.4,0.4)
      (left tip)--++(0.3,0.4)
      (right tip)--++(-0.4,0.4)
      (right tip)--++(0.3,0.4);
    \node at (0,-0.65) {$\treename$};
    \ifnum\panel=0
      \node[below right] at (root) {$v_0$};
      \node[right] at (fork) {$v_1$};
      \node[right] at (0,0.6) {$e_0$};
      \node[left] at (0,0.6) {$\frac16$};
      \node[below left] at (-0.55,1.75) {$\frac32$};
    \else
      \coordinate (left new) at (-0.63,1.83);
      \coordinate (right new) at (0.42,1.78);
      \draw[new branch] (left new)--++(-0.35,-0.3);
      \draw[new branch] (right new)--++(0.34,-0.25)--++(0.27,0)
        (0.76,1.53)--++(0,-0.3);
      \node[above right] at (left new) {$v_2$};
      \node[above left] at (right new) {$v_3$};
      \node[below left] at (-0.27,1.47) {$\frac12$};
      \node[below right] at (0.17,1.43) {$\frac14$};
    \fi
  \end{scope}
}
\end{tikzpicture}
    \caption{Extending trees in Example \ref{exunique}}
    \label{fig:uniqueextension}
\end{figure}
\begin{table}[ht]
    \centering
\begin{tabular}{ |p{1.8cm}|p{3.7cm}|p{3.7cm}|  }
\hline
Vertices & $\mathcal{T}_1$ & $\mathcal{T}_2$  \\
\hline
\hline
$v_0$ & $\left( 0, \frac{1}{x^{3}} \right)$ & $\left( 0, \left(\frac{1}{x^{3}}, 0 \right) \right)$\\
\hline
$v_1$ & $\left( \frac{1}{2}, \frac{dx}{x^2(x+1)^2} \right)$ & $\left( 1, \frac{dx}{x^{3}(x+1)^{4}} \right)$\\
\hline
$v_{2}$ & $\left( 1, \frac{dx}{x^{2}} \right)$ & $\left( 2, \frac{dx}{x^{2}(x+1)} \right)$  \\
\hline
$v_{3}$ & $\left( \frac{3}{4}, \frac{dx}{x^{2}} \right)$ & $\left( \frac{7}{4}, \frac{dx}{x^{2}(x+1)^{2}} \right)$  \\
\hline
\end{tabular}
    \caption{Degeneration data of the trees in Figure \ref{fig:uniqueextension}}
    \label{tab:degenerationdataexunique}
\end{table}
\end{example}

\begin{remark}
    \label{remarkminimal} In the context of Proposition \ref{propcriterionhurwitzlift}, if the ramification jumps of $\phi_n$ are $(m_1, \ldots, m_n)$ and $m_n = p m_{n-1}$, then the depth at a place on the trunk of $\mathcal{T}_n$ is $p$ times the depth at the corresponding place on $\mathcal{T}_{n-1}$. This relation motivates the general construction below.
\end{remark}

\section{Examples and auxiliary branches}
\label{secvanishofHurwitztree}
\label{secex}
The next example illustrates extensions with minimal and nonminimal final jumps.

\begin{example}
\label{exp3}
Let $p=3$, and let the characters $\overline{\chi}_{2,\min}$ and $\overline{\chi}_2$ in $\cohom^1(k(x),\mathbb Z/9)$ be represented by the respective Witt vectors
\[
\begin{aligned}
 \overline{\chi}_{2,\min}:&\quad
 \left(\frac{2}{x^7},\frac{1}{x^{19}}+\frac{1}{x}\right),\\
 \overline{\chi}_2:&\quad
 \left(\frac{2}{x^7},\frac{2}{x^{26}}+\frac{1}{x^5}\right).
\end{aligned}
\]
Their upper-jump sequences are $(7,21)$ and $(7,26)$, respectively; since $21=3\cdot7$, the first final jump is minimal and the second is nonminimal. Let $\overline\chi_1$ be their common order-$3$ quotient, represented by $2/x^7$. Assume that a lift $\chi_1$ of $\overline\chi_1$ has the lower tree $\mathcal T_1$ shown on the left of Figure \ref{fig:extendingtreesdraw3}. Its data are listed in the second column of Table \ref{tab:degenerationdatachar3}, including subdivision places on the lower edges inserted for comparison with the upper trees.

Put $Q_3=x^3-x^2-1$ and $Q_1=x-1$. In characteristic $3$, both polynomials are squarefree and have no root at $0$. The following identities verify several critical differentials in Table \ref{tab:degenerationdatachar3}:
\[
\mathcal C\!\left(\frac{c^3\,dx}{x^3Q_3}\right)
=\frac{c^3\,dx}{x^3Q_3}+\frac{c\,dx}{x^3}
\quad(c=\pm1),\qquad
\mathcal C\!\left(\frac{dx}{x^2Q_1}\right)
=\frac{dx}{x^2Q_1}+\frac{dx}{x^2}.
\]
The two upper trees add $14$ and $19$ order-$3$ leaves, respectively,
matching the jump differences $21-7$ and $26-7$.

We choose the unspecified depths so that the lengths of the old edges are preserved. Put $a=\delta_1(v_1)$ and $b=\delta_1(v_2)$,
where $0<a<b<1/2$. For the lower depths, take
\begin{align*}
\delta_1(v_{0,1})&=a/6,&
\delta_1(v_{3,2})&=(4a+1)/6,&
\delta_1(v_{3,3})&=(2a+1)/4,\\
\delta_1(v_{2,2})&=(4b+1)/6,&
\delta_1(v_{2,4})&=(4b+1)/6.
\end{align*}
For the upper depths, take
\begin{align*}
\delta_2(v_{0,1})&=13a/21,&
\delta_{2,\min}(v_{3,2})&=(5a+2)/3,&
\delta_2(v_{3,3})&=(10a+7)/8,\\
\delta_{2,\min}(v_{2,2})&=(5b+2)/3,&
\delta_2(v_{2,4})&=(8b+5)/6.
\end{align*}
The resulting thicknesses on every subdivided old edge are positive and
sum to its lower thickness.

The lower tree has exact differentials below depth $3/2$, logarithmic terminal differentials at depth $3/2$, and marked leaves of conductor $1$; its root coefficient is $-7\cdot2=1$ in characteristic $3$. Figure
\ref{fig:extendingtreesdraw3} depicts extensions $\mathcal T_{2,\min}$
and $\mathcal T_2$ of this lower tree.

In both upper trees, vertices and leaves corresponding to the lower tree have group order $9$, while new auxiliary components and new leaves have group order $3$. The equidistant branch attached at $v_{0,1}$ in $\mathcal T_2$ ends at $v_{0,2}$, whose data are listed in the table, and carries six order-$3$ leaves. For terminal components on new equidistant branches that are omitted from the table, Proposition \ref{propequidistanttree} specifies their depths, differentials, and numbers of marked leaves.
\begin{figure}[ht]
    \centering
\resizebox{\textwidth}{!}{%
\begin{tikzpicture}[
    x=1cm,y=1cm,line cap=round,line join=round,
    old branch/.style={draw=blue!70!cyan,line width=0.8pt},
    new branch/.style={draw=violet,line width=0.8pt},
    every node/.style={font=\small,inner sep=2pt}
]
\foreach \panel/\treename in {0/{\mathcal T_1},1/{\mathcal T_{2,\min}},2/{\mathcal T_2}} {
    \begin{scope}[xshift=4.7*\panel cm]
        \coordinate (root) at (0,0);
        \coordinate (fork) at (0,1.45);
        \coordinate (left fork) at (-1.25,2.9);
        \coordinate (left tip) at (-2.18,3.83);
        \coordinate (middle tip) at (-0.5,3.75);
        \coordinate (right tip) at (1.2,3.8);
        \draw[old branch] (root)--(fork)--(left fork)--(left tip);
        \draw[old branch] (left fork)--(middle tip);
        \draw[old branch] (fork)--(right tip);
        \draw[old branch] (left tip)--++(-0.4,0.4)
            (left tip)--++(0.4,0.4);
        \draw[old branch] (middle tip)--++(-0.37,0.45)
            (middle tip)--++(0,0.47)
            (middle tip)--++(0.43,0.48);
        \draw[old branch] (right tip)--++(-0.43,0.45)
            (right tip)--++(0,0.46)
            (right tip)--++(0.25,0.47);
        \node at (-0.45,-0.72) {$\treename$};
        \node[right] at (right tip) {$v_3$};
        \ifnum\panel=0
            \node[below right] at (root) {$v_0$};
            \node[right] at (fork) {$v_1$};
            \node[below left] at (left fork) {$v_2$};
            \node[right] at (0,0.77) {$e_0$};
            \node[right] at (0.64,2.7) {$e_1$};
            \node[left] at (-0.64,2.12) {$e_2$};
            \node[above] at (-2.58,4.23) {$0$};
        \else
            \coordinate (left critical) at (-1.72,3.37);
            \coordinate (middle critical) at (-0.86,3.34);
            \coordinate (right critical) at (0.85,3.11);
            \draw[new branch] (right critical)--++(0.4,0)
                (right critical)--++(0.4,-0.2)
                (right critical)--++(0.22,-0.34);
            \node[above right] at (right critical) {$v_{3,1}$};
            \ifnum\panel=1
                \draw[new branch] (left critical)--++(-0.48,-0.12)
                    (left critical)--++(-0.23,-0.43);
                \node[above] at (-1.6,3.53) {$v_{2,3}$};
                \draw[new branch] (middle critical)--++(0.4,-0.1)
                    (middle critical)--++(0.26,-0.27)
                    (middle critical)--++(0.09,-0.4);
                \node[above right] at (-0.43,3.26) {$v_{2,1}$};
                \coordinate (middle split) at (-1.06,3.11);
                \coordinate (middle new) at (-1.25,3.3);
                \draw[new branch] (middle split)--(middle new)
                    (middle new)--++(-0.2,0.03)
                    (middle new)--++(-0.14,0.17)
                    (middle new)--++(0,0.22);
                \node[below right] at (middle split) {$v_{2,2}$};
                \coordinate (right split) at (0.53,2.49);
                \coordinate (right new) at (0.26,2.6);
                \draw[new branch] (right split)--(right new)
                    (right new)--++(-0.25,-0.07)
                    (right new)--++(-0.18,0.08)
                    (right new)--++(-0.09,0.23);
                \node[below right] at (right split) {$v_{3,2}$};
            \else
                \draw[new branch] (left critical)--++(-0.28,-0.28);
                \draw[new branch] (middle critical)--++(0.35,0.09)
                    (middle critical)--++(0.38,-0.12)
                    (middle critical)--++(0.34,-0.22)
                    (middle critical)--++(0.24,-0.3);
                \coordinate (left split) at (-1.48,3.13);
                \coordinate (left new) at (-1.73,2.88);
                \draw[new branch] (left split)--(left new)
                    (left new)--++(-0.27,0)
                    (left new)--++(-0.13,-0.13)
                    (left new)--++(0,-0.24);
                \node[above right] at (-1.54,3.19) {$v_{2,4}$};
                \coordinate (right split) at (0.42,2.28);
                \coordinate (right new) at (0.72,2.16);
                \draw[new branch] (right split)--(right new)
                    (right new)--++(0.2,0)
                    (right new)--++(0,-0.25);
                \node[above left] at (right split) {$v_{3,3}$};
                \coordinate (root split) at (0,0.65);
                \coordinate (root new) at (0.36,0.65);
                \draw[new branch] (root split)--(root new);
                \foreach \angle in {-65,-39,-13,13,39,65} {
                    \draw[new branch] (root new)--++(\angle:0.28);
                }
                \node[left] at (root split) {$v_{0,1}$};
                \node[below right] at (0.4,0.35) {$v_{0,2}$};
            \fi
        \fi
    \end{scope}
}
\end{tikzpicture}%
}
    \caption{Extending trees in Example \ref{exp3}}
    \label{fig:extendingtreesdraw3}
\end{figure}
\begin{table}[ht]
    \centering
\begin{tabular}{ |p{1.5cm}|p{3.8cm}|p{3.8cm}|p{3.8cm}|  }
\hline
Vertices & $\mathcal{T}_1$ & $\mathcal{T}_{2,\min}$ & $\mathcal{T}_2$ \\
\hline
\hline
$v_0$ & $\left( 0, \frac{2}{x^{7}} \right)$ & $\left( 0, \left(\frac{2}{x^{7}}, \frac{1}{x^{19}}+\frac{1}{x} \right) \right)$ & $\left( 0, \left( \frac{2}{x^{7}}, \frac{2}{x^{26}}+\frac{1}{x^5} \right)\right)$ \\
\hline
$v_{0,1}$ & $\left( \delta_1(v_{0,1}), \frac{dx}{x^{8}} \right)$ & $\left( 3\delta_1(v_{0,1}),  \frac{dx}{x^{22}} \right)$ & $\left( \delta_2(v_{0,1}),  \frac{2dx}{x^{21}(x-1)^{6}} \right)$ \\
\hline
$v_{0,2}$ & --- & --- & $\left( \frac{3}{2},  \frac{-dx}{x(x^5-1)} \right)$ \\
\hline
$v_1$ & $\left( \delta_1(v_1), \frac{dx}{x^{5}(x-1)^{3}} \right)$ & $\left( 3\delta_1(v_1), \frac{dx}{x^{13}(x-1)^{9}} \right)$ & $\left( 3\delta_1(v_1), \frac{-dx}{x^{13}(x-1)^{8}} \right)$ \\
\hline
$v_{3}$ & $\left( \frac{3}{2}, \frac{dx}{x(x^2+1)} \right)$ & $\left( \frac{5}{2}, \frac{-dx}{x(x^2+1)} \right)$ & $\left( \frac{5}{2}, \frac{-dx}{x(x^2+1)} \right)$ \\
\hline
$v_{3,1}$ & $\left( \frac{1}{2}, \frac{dx}{x^3} \right)$ & $\left( \frac{3}{2}, \frac{dx}{x^{3}(x^3-x^2-1)} \right)$ & $\left( \frac{3}{2}, \frac{dx}{x^{3}(x^3-x^2-1)} \right)$ \\
\hline
$v_{3,2}$ & $\left( \delta_1(v_{3,2}), \frac{dx}{x^3} \right)$ & $\left( \delta_{2,\min}(v_{3,2}), \frac{dx}{x^{6}(x+1)^3} \right)$ & --- \\
\hline
$v_{3,3}$ & $\left( \delta_1(v_{3,3}), \frac{dx}{x^{3}} \right)$ & --- & $\left( \delta_2(v_{3,3}), \frac{-dx}{x^{6}(x-\sqrt{2})^{2}} \right)$ \\
\hline
$v_{2}$ & $\left( \delta_1(v_{2}), \frac{-dx}{x^{2}(x-1)^3} \right)$ & $\left(  3\delta_1(v_2), \frac{-dx}{x^4(x-1)^{9}} \right)$ & $\left( 3\delta_1(v_2), \frac{-dx}{x^{6}(x-1)^{7}} \right)$ \\
\hline
$v_{2,1}$ & $\left( \frac{1}{2}, \frac{-dx}{x^3} \right)$& $\left( \frac{3}{2}, \frac{-dx}{x^{3}(x^3-x^2-1)} \right)$ & $\left( \frac{3}{2} , \frac{-dx}{x^{3}(x^4-1)} \right)$ \\
\hline
$v_{2,2}$ & $\left( \delta_1(v_{2,2}), \frac{-dx}{x^3} \right)$ & $\left( \delta_{2,\min}(v_{2,2}), \frac{-dx}{x^{6}(x-2)^{3}} \right)$ & $\left( 3\delta_1(v_{2,2}), \frac{-dx}{x^7} \right)$ \\
\hline
$v_{2,3}$ & $\left( \frac{1}{2}, \frac{dx}{x^2} \right)$ & $\left( \frac{3}{2}, \frac{dx}{x^{2}(x^2-1)} \right)$ & $\left( \frac{3}{2}, \frac{dx}{x^{2}(x-1)} \right)$ \\
\hline
$v_{2,4}$ & $\left( \delta_1(v_{2,4}), \frac{dx}{x^2} \right)$ & $\left( 3\delta_1(v_{2,4}), \frac{dx}{x^{4}} \right)$ & $\left( \delta_2(v_{2,4}), \frac{dx}{x^{3}(x-2)^{3}} \right)$ \\
\hline
\end{tabular}
    \caption{Degeneration data of the trees in Figure \ref{fig:extendingtreesdraw3}. A dash denotes a vertex absent from that tree.}
    \label{tab:degenerationdatachar3}
\end{table}
\end{example}

\begin{remark}
    In the above example, at $v_0$, $v_1$, $v_2$, and the critical vertices of lower depth $1/2$, the corresponding depth in both $\mathcal{T}_{2,\min}$ and $\mathcal{T}_2$ is three times the lower depth. At other inserted vertices, the upper depth may be strictly greater than three times the lower depth.
\end{remark}

\begin{samepage}
An \emph{equidistant tree with $\mathbb Z/p$-leaves} consists of a trunk ending at a single component carrying all its marked leaves.
\end{samepage}

\begin{proposition}
\label{propequidistanttree}
An outgoing direction from $v$ with slope $l\geq1$, $p\nmid l$, coefficient $a\in k^\times$, and $\delta_{\mathcal T}(v)<p/(p-1)$ admits a compatible equidistant continuation with $l+1$ order-$p$ leaves. The terminal component has group $\mathbb Z/p$. In a coordinate with the incoming point at infinity, its trunk $e$ and terminal differential are
\[
 \epsilon_e=\frac{p/(p-1)-\delta_{\mathcal T}(v)}{l},\qquad
 \delta_{\mathcal T}(t(e))=\frac p{p-1},\qquad
 \omega_{\mathcal T}(t(e))=\frac{a\,dx}{x(x^l-a)}.
\]
\end{proposition}
\begin{proof}
The thickness is positive and gives the required depth increment. The identity
\[
 \frac{a\,dx}{x(x^l-a)}=\frac1l\,d\log(1-a/x^l)
\]
shows that the differential is logarithmic, since $1/l\in\mathbb F_p$. It has $l+1$ distinct simple poles, no finite zeros, order $l-1$ at infinity, and infinity coefficient $a$, giving the required slope and compatibility.
\end{proof}

\section{Proof of the vanishing theorem}
\label{sectechnical}
We first construct differentials satisfying the Cartier equation at a critical vertex. We then use these differentials to extend the lower tree, one branch at a time. Throughout this section, $k$ is algebraically closed of characteristic $p$.

\subsection{Two Cartier constructions}
\label{secsolutioncartier}
Write $\omega=c\,dx/D$, where $c\in k^\times$ and $D$ is monic of degree $m+1$, with $m\geq1$. At a critical vertex we need a differential $\Omega$ with no finite zeros and with $\mathcal C(\Omega)=\Omega+\omega$. The \emph{incoming slope} of $\Omega$ is $\ord_\infty(\Omega)+1$. A \emph{minimal solution} has incoming slope $pm$; a \emph{subminimal solution} has incoming slope $M<pm$.

\begin{proposition}[Minimal solution]
\label{theoremminimalcartiergeneral}
There is a monic squarefree polynomial $Q$, relatively prime to $D$, of degree $(p-1)m$ such that
\[
 \Omega=\frac{c^p\,dx}{DQ},\qquad
 \mathcal C(\Omega)=\Omega+\omega,\qquad
 \ord_\infty(\Omega)=pm-1.
\]
\end{proposition}
\begin{proof}
For a nonzero polynomial $H$, put
\[
 \Theta(H)=-H^p\partial^{p-1}(1/H),\qquad \partial=d/dx.
\]
The denominator of $\partial^{p-1}(1/H)$ divides $H^p$, so $\Theta(H)$ is a polynomial. Since $\partial^p=0$, its derivative is zero. It is therefore the $p$th power of a unique polynomial $\Phi(H)$.

We need two properties of $\Phi$. First, write $H=\sum_{j=0}^t h_jx^j$. Each coefficient of $\Theta(H)$ is a homogeneous polynomial of degree $p-1$ in $h_0,\ldots,h_t$, with coefficients in $\mathbb F_p$. Put $b_j=h_j^{1/p}$. Taking $p$th roots in $\Theta(H)=\Phi(H)^p$ shows that each coefficient of $\Phi(H)$ is a polynomial of total degree at most $p-1$ in $b_0,\ldots,b_t$.

Second, if $H$ is monic of degree $pm+1$, expansion at infinity and $(p-1)!=-1$ give
\[
 \Theta(H)=x^{p(p-1)m}+\text{lower-degree terms}.
\]
Thus $\Phi(H)$ is monic of degree $(p-1)m$.

Set $N=(p-1)m$ and write $Q=x^N+\sum_{j<N}u_j^p x^j$. To obtain the required Cartier relation, we seek $Q$ satisfying
\begin{equation}
\label{eqnminimalcartierfixedpoint}
 \Phi(DQ)=Q+c^{p-1}
\end{equation}
This is a system $u_j^p-R_j(u)=\beta_j$ for $0\leq j<N$, where $R_j(u)$ is the coefficient of $x^j$ in $\Phi(DQ)$, $\beta_0=-c^{p-1}$, and $\beta_j=0$ for $j>0$. Since $D$ is fixed, the $p$th roots of the coefficients of $DQ$ are affine-linear in the $u_j$, so the coefficient bound above gives $\deg R_j\leq p-1$.

To solve it, consider the polynomial map with coordinates $u_j^p-R_j(u)$. Its homogenization defines a morphism $\mathbb P^N\to\mathbb P^N$: at infinity the coordinates are $u_j^p$, so they do not vanish simultaneously. The pullback of $\mathcal O(1)$ is $\mathcal O(p)$, which makes the morphism finite and surjective. Points at infinity map to infinity, so an affine point has an affine preimage. Thus a monic $Q$ satisfying \eqref{eqnminimalcartierfixedpoint} exists.

With $H=DQ$, this equation gives
\[
 c^p\Theta(H)=(c^p+cQ)^p.
\]
The identity $\mathcal C(f\,dx)=g\,dx\iff\partial^{p-1}f=-g^p$ now yields $\mathcal C(\Omega)=\Omega+\omega$. Finally,
\[
 \Theta(H)\equiv(H')^{p-1}\pmod H.
\]
Indeed, in the numerator of $-H^p\partial^{p-1}(1/H)$, the only term not divisible by $H$ is $(H')^{p-1}$. At a root $z$ of $Q$ we obtain
\[
 c^p\bigl(D(z)Q'(z)\bigr)^{p-1}=c^{p^2}\ne0.
\]
Thus $Q$ is squarefree and has no common root with $D$. The formula for $\Omega$ shows that it has no finite zeros and has order $pm-1$ at infinity.
\end{proof}

\begin{lemma}[Congruent subminimal solution]
\label{lemcongruentsubminimalcartier}
Suppose $\omega=c\,dx/D$ is Cartier-nilpotent. There is a squarefree polynomial $Q$, relatively prime to $D$, such that
\[
 \Omega=\frac{c^p\,dx}{DQ},\qquad \mathcal C(\Omega)=\Omega+\omega,
\]
and, with $M=m+\deg Q$,
\[
 m\leq M\leq pm-1,\qquad M+1\equiv0\pmod p.
\]
In particular, $\Omega$ has no finite zeros and $\ord_\infty(\Omega)=M-1$. Here $Q$ need not be monic.
\end{lemma}
\begin{proof}
We first solve the Cartier equation, then add a logarithmic differential to obtain a solution with constant numerator. Choose $s$ with $\mathcal C^s(\omega)=0$ and set
\[
 \eta=-\sum_{j=0}^{s-1}\mathcal C^j(\omega)=\frac{E\,dx}{D},
 \qquad f=E/D,\qquad T=\partial+f.
\]
Here $T$ acts on $k(x)$ by $T(F)=F'+fF$.
Cartier introduces no new poles and does not increase their orders. Thus $\eta=E\,dx/D$ with $E\in k[x]$. Since $\omega$ is regular at infinity and Cartier preserves regularity, so is $\eta$; hence $\deg E\leq m-1$. The sum defining $\eta$ gives $\mathcal C(\eta)=\eta+\omega$. Jacobson's identity for a derivation plus multiplication is
\[
 (\partial+f)^p=\partial^p+\partial^{p-1}(f)+f^p,
\]
where the last two terms act by multiplication. Together with the Cartier formula, this gives
\[
 T^p(F)=-\frac{c^p}{D^p}F\qquad(F\in k(x)).
\]
Let $r$ be the least nonnegative residue of $m+1$ modulo $p$ and choose a divisor $R\mid D$ of degree $r$. This choice makes $\deg(D^{p-1}R)$ divisible by $p$. Put
\[
 h=D^{p-1}R,\qquad g=-T^{p-1}h.
\]
Then
\begin{equation}
\label{eqnsubminimalcartieroperatorlinear}
 Dg'+Eg=c^pR.
\end{equation}
For $P\in k[x]$ and $j\geq1$, the formula $T(D^jP)=D^{j-1}(jD'P+DP'+EP)$ shows that $g$ is polynomial.

We next show that $g=RQ$ for a polynomial $Q$ having no common root with $D$. A simple root of $D$ would give $\omega$ a nonzero residue, which could not vanish under any Cartier iterate. Thus every root of $D$ is multiple. At a root $a$ of $D$, put $l=\ord_aD\geq2$. The pole order of $\partial^{p-1}f$ is at most $l+p-1<pl$. Multiplying $\partial^{p-1}f+f^p=-c^p/D^p$ by $D^p$ and evaluating at $a$ therefore gives $E(a)^p=-c^p$, hence $E(a)=-c$.

Put $t=\ord_a R\geq0$. Since $E(a)=-c\ne0$, we have $\ord_a(E/D)=-l$. For $\ord_aF=q$, the terms $(E/D)F$ and $F'$ have orders $q-l$ and at least $q-1$, respectively. Since $l\geq2$, $\ord_a(TF)=q-l$. Starting from $\ord_a h=(p-1)l+t$ and applying this identity repeatedly gives
\[
 \ord_a(T^j h)=(p-1-j)l+t\qquad(0\leq j\leq p-1).
\]
Thus $\ord_a g=\ord_a R$ at every root of $D$, so $g=RQ$ with $Q\in k[x]$ and $(Q,D)=1$. Dividing \eqref{eqnsubminimalcartieroperatorlinear} by $Dg$ gives
\[
 \eta+dg/g=\frac{c^p\,dx}{DQ}=\Omega.
\]
Since Cartier fixes the logarithmic differential $dg/g$, we have $\mathcal C(\Omega)=\Omega+\omega$. If $Q(z)=0$, equation \eqref{eqnsubminimalcartieroperatorlinear} gives $D(z)Q'(z)=c^p$, proving that $Q$ is squarefree.

It remains to bound the degree of $Q$ and prove the congruence. Put $A=\deg h=(p-1)(m+1)+r$. Since $p\mid A$ and $\deg E-\deg D\leq-2$, the first application of $T$ lowers degree by at least two, and each subsequent application lowers it by at least one. Hence
\[
 \deg g\leq A-p,\qquad 0\leq\deg Q\leq(p-1)m-1.
\]
If $p\nmid\deg g$, the term $Dg'$ in \eqref{eqnsubminimalcartieroperatorlinear} has degree $m+\deg g$, strictly greater than both $\deg(Eg)$ and $\deg R$. This is impossible. Therefore $p\mid\deg g=\deg Q+r$, and $M+1=m+\deg Q+1\equiv0\pmod p$.
\end{proof}

\subsection{The recursive extension problem}
\label{secbranchrootdata}
To prove Theorem \ref{theoremmain}, we extend the prescribed lower tree one vertex at a time. The upper differential chosen at a vertex prescribes the root principal parts for its outgoing branches. We therefore need an extension statement for these branch problems as well as for the original depth-zero root.

At a nonroot vertex, the differential is the sum of its finite principal parts, with no polynomial part, by Corollary \ref{cordiffgood} or the Hurwitz-tree condition at infinity. For an edge $e$, let $\mathcal T(e)$ consist of $e$ and all components and marked leaves at or beyond $t(e)$, rooted at $s(e)$. Retain the depths, differentials, edge lengths, and marked-leaf data, except for the root modifications below. At a positive-depth root, keep only the principal part in the $e$-direction and translate the point corresponding to this direction to $x=0$. At depth zero, retain the root Witt vector.

Give the branch root the group $H=G_{\mathcal T}(t(e))$. If the current ambient group is $G\cong\mathbb Z/p^\ell$ and $|H|=p^m$, let $q_i:G\to\mathbb Z/p^i$ be the quotient map at level $i$. The image $q_i(H)$ is trivial for $i\leq\ell-m$; the retained level $i=\ell-m+j$ has image of order $p^j$. Relabel this retained quotient as local level $j$, for $1\leq j\leq m$. Restrict the data at each retained level to the branch.

An \emph{inherited branch problem} asks for an extension of the lower data obtained by these restrictions and reindexings, with upper root data prescribed as in Proposition \ref{proprecursiveextension}. We allow repeated restriction and subdivision at rational interior places, provided the root depth at the highest retained lower level is less than $1/(p-1)$. Taking principal parts commutes with Cartier and retains the leading polar coefficient in the chosen coordinate.

\begin{lemma}[Subdivision]
\label{lemprestablesubdivision}
Let $e$ be an edge of the prescribed lower tree or of the lower tree of an inherited branch problem, with thickness $\epsilon_e$, slope $d$, and coefficient $c$. At rational distance $0<\rho<\epsilon_e$ from $s(e)$, insert an unmarked vertex with depth $\delta_{\mathcal T}(s(e))+d\rho$, group $G_{\mathcal T}(t(e))$, and differential $c\,dx/x^{d+1}$. The two segments have lengths $\rho$ and $\epsilon_e-\rho$ and slope $d$. With the endpoint data unchanged, this gives a subdivision of the original tree.
\end{lemma}
\begin{proof}
The new differential has orders $d-1$ at infinity and $-d-1$ at zero. Its infinity and leading polar coefficients are both $c$, and the two depth increments sum to the original increment. At each retained quotient level, the branch points inside the disc represented by the interior place lie in one residue direction. Corollary \ref{cordiffgood} and Theorem \ref{theoremcompatibilitydiff} therefore give the stated monomial differential with that level's edge coefficient. These are the degeneration data of the lower lift, so the quotient relations in Theorem \ref{theoremCartierprediction} are preserved, also on inherited branches.
\end{proof}

The depth-zero case of the following proposition gives Theorem \ref{theoremmain}.

\begin{proposition}
\label{proprecursiveextension}
Let $\mathcal T_{n-1}$ be the compatible tree of the prescribed lower lift or the lower tree of an inherited branch problem, with its retained quotient tower and root depth $0\leq\alpha<1/(p-1)$ at the highest lower level. Here $n-1\geq1$ is the local tower length after reindexing. If $\alpha=0$, prescribe a reduced length-$n$ root Witt vector extending the lower root type. If $\alpha>0$, put $d=-\ord_0\omega_{n-1}(v_0)-1$ and prescribe a finite principal part $\Omega_0=\sum_{j\geq2}a_j\,dx/x^j$ satisfying
\[
 \mathcal C(\Omega_0)=\omega_{n-1}(v_0),\qquad
 -\ord_0\Omega_0-1=N,
\]
where $N=pd$, or $N>pd$ and $p\nmid N$. Then there is a compatible extending tree with the prescribed root data and root depth $p\alpha$.
\end{proposition}
At depth zero, let $d$ and $N$ be the final upper jumps of the lower and extending root types. We call the root data \emph{minimal} if $N=pd$, and \emph{nonminimal} if $N>pd$. At positive depth, if $\omega_{n-1}(v_0)=\sum_i b_i\,dx/x^i$, then $\sum_i b_i^p\,dx/x^{p(i-1)+1}$ has the required Cartier image and minimal slope. For nonminimal slope $N$, add a nonzero multiple of $dx/x^{N+1}$, which is exact. The proposition allows any principal part satisfying its Cartier and pole-order conditions, including those arising in the induction.

\subsection{Construction of the extending tree}
\label{secproofvanishing}
\begin{proof}[Proof of Proposition \ref{proprecursiveextension}]
Write $\delta_j$ and $\omega_j$ for the depth and differential data of $\mathcal T_j$. Throughout the proof, ``old'' means inherited from $\mathcal T_{n-1}$. We first contract any inessential subdivision vertices of the lower tree.

Every root-to-leaf path contains exactly one place where the lower depth equals $1/(p-1)$: the root depth is less than $1/(p-1)$, slopes are positive, and the vertices carrying marked leaves have depth at least $p/(p-1)$ by Proposition \ref{propmixedtree}\ref{propmixedtree1}. To define the induction parameter, form an auxiliary tree $T_{\rm crit}$ by inserting a vertex at each such critical place and retaining only the portions from the root to these vertices. Let $V_{\rm crit}$ be its set of critical vertices and define
\[
 h(\mathcal T_{n-1})=\max_{v\in V_{\rm crit}}
 d_{T_{\rm crit}}(v_0,v),
\]
where $d_{T_{\rm crit}}$ is graph distance, counting edges. Thus $h\geq1$, and $h=1$ precisely when the critical place lies on the trunk, including the target of the trunk. We use strong induction on $h$, simultaneously for all inherited branch problems and all lower tower lengths $n-1\geq1$.

Let $v_1=t(e_0)$, let $d$ and $\epsilon$ be the lower trunk slope and thickness, and put $\alpha=\delta_{n-1}(v_0)$ and $\beta=\delta_{n-1}(v_1)$. Thus $\beta=\alpha+d\epsilon$. Let $N$ be the prescribed slope of the upper trunk at the root. At depth zero it is the final upper jump of the prescribed Witt vector, so again $N=pd$ or $N>pd$ with $p\nmid N$.

In the nonminimal case, let $\gamma\ne0$ be the coefficient of $dx/x^{N+1}$ in $\Omega_0$ if the root depth is positive; if the root depth is zero, put $\gamma=-N d_N$, where the last Witt coordinate is $d_Nx^{-N}+\cdots$.

For every old vertex or marked leaf with group of order $p^i$, give the corresponding upper vertex or leaf a group of order $p^{i+1}$. Every new vertex and marked leaf on an attached auxiliary branch has group $\mathbb Z/p$. If we insert an upper vertex on an old edge, we also insert the corresponding lower subdivision vertex using Lemma \ref{lemprestablesubdivision}. The group of the inserted upper vertex is assigned by the same rule as for the old vertices.

\smallskip\noindent\emph{Base case: $h=1$.}
Here $\beta\geq1/(p-1)$, so the critical place lies on the trunk. If $\beta>1/(p-1)$, first subdivide the trunk at its unique critical place, at distance $\rho=(1/(p-1)-\alpha)/d$ from the root. By Lemma \ref{lemprestablesubdivision}, the lower differential at the inserted vertex is $c\,dx/x^{d+1}$. Relabel this vertex as $v_1$ and take the segment from the root to it as the trunk, with $\epsilon$ now denoting its thickness. We thus reduce to $\beta=1/(p-1)$. Write
\[
 \omega_{n-1}(v_1)=\frac{c\,dx}{D(x)},\qquad \deg D=d+1,
\]
with $D$ monic. By the depth inequalities in Theorem \ref{theoremCartierprediction}, successive lower quotient depths satisfy $\delta_{j-1}\leq\delta_j/p$. Cartier therefore either gives zero or passes to the next lower differential, eventually giving zero at or before the order-$p$ level, whose depth is below $p/(p-1)$. Thus $\omega_{n-1}(v_1)$ is Cartier-nilpotent.

If $N=pd$, Proposition \ref{theoremminimalcartiergeneral} gives a critical differential $c^p\,dx/(DQ)$ of incoming slope $pd$, with $Q$ monic of degree $(p-1)d$. Set $\delta_n=p\delta_{n-1}$ pointwise on the trunk and attach a new order-$p$ leaf at every root of $Q$.

If $N>pd$, Lemma \ref{lemcongruentsubminimalcartier} gives a critical differential $c^p\,dx/(DQ)$ of incoming slope $M$, where
\[
 d\leq M<pd,\qquad p\mid M+1,\qquad \deg Q=M-d.
\]
Let $a\ne0$ be the infinity coefficient of the critical differential and put $L=N-M$. Then $N>pd>M\geq1$ and $L\not\equiv1\pmod p$. Split the trunk at $v'$ with
\begin{equation}
\label{eqncriticalsplitthickness}
 \epsilon_{\rm in}=\frac{pd-M}{N-M}\epsilon,
 \qquad
 \epsilon_{\rm out}=\frac{N-pd}{N-M}\epsilon.
\end{equation}
The first segment, from the root to $v'$, has length $\epsilon_{\rm in}$ and slope $N$; the second, from $v'$ to $v_1$, has length $\epsilon_{\rm out}$ and slope $M$. Their lengths sum to $\epsilon$, and $N\epsilon_{\rm in}+M\epsilon_{\rm out}=pd\epsilon$. Thus the upper depth starts at $p\alpha$ and ends at $p\beta$. In the interior it is strictly greater than $p\delta_{n-1}$: the difference $\delta_n-p\delta_{n-1}$ increases along the first segment and decreases to zero along the second. To match the coefficients, choose $b\ne0$ with $\gamma/(-b)^L=a$ and put
\begin{equation}
\label{eqncriticalsplitdifferential}
 \omega_n(v')=\frac{\gamma\,dx}{x^{M+1}(x-b)^L}.
\end{equation}
At $v'$, the incoming slope is $N$. The outgoing slope at $x=0$, toward $v_1$, is $M$, and the slope in the new direction $b$ is $L-1$. Its infinity coefficient is $\gamma$, and its leading polar coefficient at $0$ is $\gamma/(-b)^L=a$, as required for compatibility at the root and at $v_1$. In this nonminimal case, \eqref{eqncriticalsplitdifferential} is exact because $x^{-(M+1)}$ is a $p$th power and $\mathcal C(dx/(x-b)^L)=0$.

In the direction $b$, use the local coordinate $u=x-b$. The coefficient of $du/u^L$ in the principal part of $\omega_n(v')$ is $c_b=\gamma/b^{M+1}$. Attach the equidistant tree of Proposition \ref{propequidistanttree} with slope $L-1$ and coefficient $c_b$; it has $L$ order-$p$ leaves. The attached equidistant trunk has positive length because $\delta_n(v')<p\beta\leq p/(p-1)$.

At the critical vertex, attach a new order-$p$ leaf for each root of $Q$ as well.

In both cases preserve the old branches beyond the critical place, with their coordinates, lengths, and markings; increase their group orders as specified above. On these branches, where $\delta_{n-1}>1/(p-1)$, set
\[
 \delta_n=\delta_{n-1}+1,\qquad \omega_n=-\omega_{n-1}.
\]
At every pole of $\omega_{n-1}(v_1)$, the upper and lower differentials have the same pole order, at least two. Since Cartier strictly lowers this order, the relation $\mathcal C(\omega_n)=\omega_n+\omega_{n-1}$ forces their leading polar coefficients to be opposite, giving compatibility with the continued branches.

The new poles of the differential at the critical vertex are the roots of $Q$. They are simple and distinct from the old poles. In the minimal case there are $\deg Q=(p-1)d=N-d$ new leaves. In the nonminimal case the auxiliary branch attached at $v'$ contributes $L$ leaves and the critical vertex contributes $\deg Q$, for a total of $L+\deg Q=(N-M)+(M-d)=N-d$. 

\smallskip\noindent\emph{Induction step: $h>1$.}
Assume the proposition for every inherited branch problem of height less than $h$. The first vertex $v_1$ is then subcritical, so $\beta<1/(p-1)$. By Proposition \ref{propmixedtree}\ref{propmixedtree1}, no marked leaf is attached to $v_1$. After contracting inessential subdivisions, its component is stable. With one incoming special point and no attached leaves, stability gives at least two outgoing nodes. Let $e_1,\ldots,e_s$ be the outgoing edges at $v_1$. Each branch $\mathcal T_{n-1}(e_i)$ is rooted at $v_1$; its paths to critical vertices omit the original trunk edge. Consequently,
\[
 h(\mathcal T_{n-1}(e_i))\leq h(\mathcal T_{n-1})-1<h.
\]
Let $a_i$ be the coordinate of the node corresponding to $e_i$, and put $l_i=d_{\mathcal T_{n-1}}(e_i)+1$, the pole order of the lower differential at $a_i$. Work temporarily in a translated coordinate with $a_1=0$, and express the resulting differentials in the prescribed coordinate before assigning them to $v_1$. Write
\[
 \omega_{n-1}(v_1)=\frac{c\,dx}{\prod_i(x-a_i)^{l_i}},
 \qquad \sum_i l_i=d+1.
\]
Set $\delta_n(v_1)=p\beta$. If $N=pd$, choose
\begin{equation}
\label{eqnminimalbranchform}
 \omega_n(v_1)=
 \frac{c^p\,dx}{x^{pl_1-p+1}\prod_{i>1}(x-a_i)^{pl_i}}.
\end{equation}
It has incoming slope $pd$ and Cartier image $\omega_{n-1}(v_1)$: write it as $\bigl(c/\prod_i(x-a_i)^{l_i}\bigr)^p x^{p-1}dx$. Its outgoing slopes are $p(l_1-1)$ and $p(l_i-1)+p-1$ for $i>1$.

We extend each outgoing branch $\mathcal T_{n-1}(e_i)$. At $a_i$, use $u=x-a_i$ and write $\operatorname{PP}_i$ for the principal part in this coordinate. If $\operatorname{PP}_i(\eta)=\sum_{j\geq1}b_j\,du/u^j$, the local Cartier formula gives
\[
 \mathcal C(\operatorname{PP}_i(\eta))
 =\sum_{q\geq0}b_{pq+1}^{1/p}\frac{du}{u^{q+1}}
 =\operatorname{PP}_i(\mathcal C(\eta)),
\]
since Cartier preserves regular differentials at $u=0$. Put
\[
 \omega_{0,i}=\operatorname{PP}_i(\omega_{n-1}(v_1)),\qquad
 \Omega_{0,i}=\operatorname{PP}_i(\omega_n(v_1)).
\]
These root data satisfy $\mathcal C(\Omega_{0,i})=\omega_{0,i}$, as required in Proposition \ref{proprecursiveextension}. The same quotient-depth argument as in the base case makes $\omega_{n-1}(v_1)$ Cartier-nilpotent. Since principal parts commute with Cartier, $\omega_{0,i}$ is also Cartier-nilpotent, so its residue is zero: a nonzero residue remains nonzero under every Cartier iterate. The $q=0$ term of the Cartier formula then forces the residue of $\Omega_{0,i}$ to vanish as well. Both root data therefore contain only terms $du/u^j$ with $j\geq2$.

On each branch, use the local group $H_i=G_{\mathcal T_{n-1}}(t(e_i))$ and the inherited quotient tower defined in \S\ref{secbranchrootdata}. Reindexing changes neither the underlying branch nor its depths, so its height remains less than $h$. Write $|H_i|=p^{m_i}$ and apply the induction hypothesis to the reindexed branch tower of order $p^{m_i}$.

Identify the roots of the resulting upper branches with $v_1$, retaining its assigned ambient group $G_{\mathcal T_n}(v_1)$ and differential $\omega_n(v_1)$. Each branch root differential is precisely the principal part $\Omega_{0,i}$ of this differential in its attaching direction; the attached branch has upper local group of order $p^{m_i+1}$. The trunk needs no subdivision.

In the nonminimal case, put $\mu=N-pd>0$ and write $\mu=pq+r$, with $q\geq0$ and $1\leq r<p$. Choose
\begin{equation}
\label{eqnnonminimalbranchform}
 \omega_n(v_1)=
 \frac{c^p\,dx}
 {(-a_2)^{p-r}x^{pl_1-p+1}(x-a_2)^{pl_2-p+r}
 \prod_{i>2}(x-a_i)^{pl_i}}.
\end{equation}
Its Cartier image is again $\omega_{n-1}(v_1)$. Indeed, after extracting the factor $\bigl(c/\prod_i(x-a_i)^{l_i}\bigr)^p$, the remaining polynomial differential is
\[
 \frac{x^{p-1}(x-a_2)^{p-r}}{(-a_2)^{p-r}}\,dx;
\]
its only term surviving Cartier is $x^{p-1}dx$, with coefficient $1$. The outgoing slopes are
\[
 p(l_1-1),\qquad p(l_2-1)+r-1,\qquad
 p(l_i-1)+p-1\quad(i>2).
\]
For the branch in direction $a_i$, the lower root slope is $l_i-1$. Each upper outgoing slope is therefore either the minimal value $p(l_i-1)$ or is larger than this value and prime to $p$. Thus the same principal-part argument and induction hypothesis extend these branches. The incoming slope and coefficient at $v_1$ are
\[
 M=pd-p+r<pd,\qquad a=c^p/(-a_2)^{p-r}.
\]
Use the same trunk subdivision \eqref{eqncriticalsplitthickness}, coefficient matching \eqref{eqncriticalsplitdifferential}, and equidistant attachment as in the base case, with $L=N-M=p(q+1)$. The auxiliary trunk has positive thickness because $\delta_n(v')<p\beta<p/(p-1)$. Formula \eqref{eqncriticalsplitdifferential} is exact because $(x-b)^{-L}$ is a $p$th power and $p\nmid M$.

\smallskip\noindent\emph{Checking the remaining conditions.}
At a positive-depth root, assign the entire prescribed differential $\omega_n(v_0)=\Omega_0$ and depth $p\alpha$. Edge compatibility uses only its leading polar coefficient (Definition \ref{defncompatibility}). In the minimal case with positive root depth, the Cartier equation forces the coefficient of $dx/x^{pd+1}$ in $\Omega_0$ to be $c^p$. This equals the infinity coefficient of the chosen differential at $v_1$. Still in the minimal case, if the root has depth zero, the last reduced Witt coordinate has degree below $pd$, so it imposes no additional coefficient. In the nonminimal cases, \eqref{eqncriticalsplitdifferential} has the prescribed infinity coefficient $\gamma$, and the equation $\gamma/(-b)^L=a$ gives coefficient compatibility with the component at $v_1$.

On the open edge segments where the lower depth is below $1/(p-1)$, use the monomial differentials of Definition \ref{defsumcondsprec}, with the slope and coefficient of each edge. When the upper depth is $p$ times the lower depth, Cartier sends the upper monomial to the lower one. Where the upper depth is strictly greater than this multiple of the lower depth, the upper monomial is exact. Beyond the lower critical depth, use the continuation $\delta_n=\delta_{n-1}+1$ and $\omega_n=-\omega_{n-1}$ described in the base case.

It remains to count the leaves. Each branch conductor is its root slope plus one, equal to the sum of its leaf conductors by \eqref{eqnslopeconductor}. Old leaves have conductor $1$ by Proposition \ref{propmixedtree}\ref{propmixedtree3}, and their conductors are retained; new leaves have conductor $1$ by construction. Thus the branch conductor is its leaf count. In \eqref{eqnminimalbranchform}, the conductors of the outgoing branches sum to $pd+1$. In \eqref{eqnnonminimalbranchform}, they sum to $M+1$, and the auxiliary branch attached at $v'$ contributes another $N-M$ leaves. Thus the extending tree for the current branch problem has $N+1$ leaves in each case, compared with $d+1$ in its lower tree. At a depth-zero root, the $N-d$ new order-$p$ leaves give the final jump difference in \ref{c6rootHurwitz}. For $1\leq i\leq n-1$, the old leaves of order $p^{n-i}$ become leaves of order $p^{n-i+1}$ with unchanged conductors, preserving the sum $u_i-u_{i-1}$.

Multiplying corresponding group orders by $p$ preserves the subgroup indices in \ref{c7Hurwitz}, and newly attached branches add positive terms to its sums. An inserted upper vertex carrying a new branch has one old successor of the same group and an additional branch, so its sum is greater than $1$; an unmarked subdivision vertex has sum $1$. Each auxiliary terminal vertex has group $\mathbb Z/p$ and at least two leaves of the same group, so it also satisfies \ref{c7Hurwitz}.

Restore any lower subdivisions contracted at the start. The lengths of the upper segments over each old edge sum to its original thickness, so the resulting tree extends the prescribed lower tree in the sense of Definition \ref{defnextendhurwitz}.

This proves Proposition \ref{proprecursiveextension}. Applying it at a depth-zero root proves Theorem \ref{theoremmain}.
\end{proof}


\end{document}